\documentclass[11pt]{article}
\usepackage{geometry}
\usepackage{color}
\usepackage{float}
\usepackage{textcomp}
\usepackage{url}
\usepackage{amsthm}
\usepackage{amsmath}
\usepackage{amssymb}
\usepackage[framemethod=tikz]{mdframed}
\usepackage{multirow}
\usepackage{changes}
\usepackage{hyperref}
\usepackage{mathtools}
\usepackage{booktabs}
\usepackage{subfigure} 
\usepackage{caption} 
\usepackage{multirow}
\usepackage{multicol}
\usepackage{array}
\usepackage{graphicx}
\graphicspath{{figuras/}}
\usepackage{empheq,amsmath}
\usepackage
[ lined, 
boxruled,
commentsnumbered
]
{algorithm2e}

\DecMargin{0.1cm}

\newtheorem{theorem}{Theorem}[section]
\newtheorem{lemma}[theorem]{Lemma}

\newtheorem{remark}[theorem]{Remark}

\newcommand{\bi}{\begin{itemize}}
\newcommand{\ei}{\end{itemize}}
\newcommand{\ba}{\begin{array}}
\newcommand{\ea}{\end{array}}

\newcommand{\ftol}{f_{\mathrm{tol}}}
\begin{document}

\title{\textbf{Active-set Newton-MR methods for nonconvex optimization problems with bound constraints}\thanks{This work has been partially supported by the Brazilian agencies FAPESP (grants 2013/07375-0, 2022/05803-3, 2023/08706-1, and 2024/22384-0) and CNPq (grant 302073/2022-1).}}

\author{
    Ernesto G. Birgin\thanks{Institute of Mathematics and Statistics, University of S\~ao Paulo, Rua do Mat\~ao, 1010, Cidade Universit\'aria, 05508-090, S\~ao Paulo, SP, Brazil (e-mail: egbirgin@ime.usp.br, diaulas@ime.usp.br).}
    \and
    Geovani N. Grapiglia\thanks{Université Catholique de Louvain, ICTEAM/INMA, Avenue Georges Lema\^{\i}tre, 4-6/L4.05.01, B-1348, Louvain-la-Neuve, Belgium (e-mail: geovani.grapiglia@uclouvain.be).}
    \and
    Diaulas S. Marcondes\footnotemark[2]
}

\date{August 25th, 2025}

\maketitle

\begin{abstract}
This paper presents active-set methods for minimizing nonconvex twice-continuously differentiable functions subject to bound constraints. Within the faces of the feasible set, we employ descent methods with Armijo line search, utilizing approximated Newton directions obtained through the Minimum Residual (MINRES) method. To escape the faces, we investigate the use of the Spectral Projected Gradient (SPG) method and a tailored variant of the Cubic Regularization of Newton's method for bound-constrained problems. We provide theoretical guarantees, demonstrating that when the objective function has a Lipschitz continuous gradient, the SPG-based method requires no more than $\mathcal{O}(n\epsilon^{-2})$ oracle calls to find $\epsilon$-approximate stationary points, where $n$ is the problem dimension. Furthermore, if the objective function also has a Lipschitz continuous Hessian, we show that the method based on cubic regularization requires no more than $\mathcal{O}\left(n|\log_{2}(\epsilon)|\epsilon^{-3/2}\right)$ oracle calls to achieve the same goal. We emphasize that, under certain hypotheses, the method achieves $O(\epsilon^{3/2})$ descent within the faces without resorting to cubic regularization. Numerical experiments are conducted to compare the proposed methods with existing active-set methods, highlighting the potential benefits of using MINRES instead of the Conjugate Gradient (CG) method for approximating Newton directions.\\

\end{abstract}

\section{Introduction} \label{sec1}

The optimization problem entails finding a solution that satisfies given constraints while minimizing a specified objective function. This type of problem has far-reaching applications in various fields, including psychology, medicine, economics, engineering, physics, and biology, to name a few. In this work, we focus on cases where the functions that define the feasible region and the objective function are continuous and differentiable. Among the methods developed to address such problems, augmented Lagrangians deserve special attention. Of the existing implementations, Algencan~\cite{abmstango,bmbook,bmcomper} is one of the most efficient and robust. It solves a sequence of bound-constrained problems using Gencan~\cite{bmgencan}. When Hessians are not available or available as an operator, Gencan computes approximate Newtonian directions using conjugate gradients. Recent studies~\cite{Orban2019,Liu2021,Liu2022,Liu2023} have highlighted the theoretical properties and practical performance of Newtonian methods for unconstrained minimization, which leverage the minimal residual method (MINRES)~\cite{Paige1975} to solve linear Newtonian systems. Motivated by these results, and having in mind subproblems in augmented Lagrangian methods, in this work, we propose Newton-MR methods for bound-constrained minimization.

The development of methods for minimization with bound constraints is a vibrant and dynamic field. During the past decades, numerous methods have been devised; see, for example, \cite{abmbetra,bmgencan,Facchinei2002,asacg,Heinkenschloss1999,Lin1999,lbfgsb3} and the references therein. A comprehensive numerical comparison of freely available software packages was presented in~\cite{bg}. The methods proposed here follow the active set paradigm. Specifically, we propose two distinct extensions of the Newton-MR method from \cite{Liu2023} for minimization with bound constraints. In \cite{Liu2023}, the method achieves a worst-case iteration complexity of $\mathcal{O}(\epsilon^{-3/2})$ to find a point $x$ such that $\|\nabla f(x) \| \leq \epsilon$ when minimizing $f(x)$ over $x \in \mathbb{R}^n$. Our first proposed extension, which introduces greater flexibility in algorithmic choices, takes no more than $\mathcal{O}(n\epsilon^{-2})$ calls to the oracle to find a point $x$ such that $\| P_\Omega( x - \nabla f(x) ) - x \| \leq \epsilon$ when applied to the problem of minimizing $f(x)$ subject to $x \in \Omega$, where $\Omega$ is the feasible set defined by the bound constraints. The second extension takes no more than $\mathcal{O}(n|\log_{2}(\epsilon)|\epsilon^{-3/2})$ calls to the oracle to achieve the same goal. Numerical experiments demonstrate that the first option is more robust and efficient compared to the second, using both unconstrained and bound-constrained minimization problems from the CUTEst collection \cite{cutest}. Additional experiments also show that the first version outperforms Gencan, particularly when second derivatives are available but matrix factorizations are not.

The remainder of this paper is organized as follows. The two Newton-MR-based methods for bound-constrained minimization and their theoretical properties are described in Section~\ref{sec2}. Numerical experiments are presented and analyzed in Section~\ref{sec3}. Conclusions and directions for future work are presented in the final section.

\section{Problem definition and new methods} \label{sec2}

Consider the problem
\[
\text{Minimize $f(x)$ subject to $x \in \Omega$,}
\]
\noindent where $\Omega=\left\{x\in\mathbb{R}^{n}\,:\,\ell\leq x\leq u\right\}$, with $\ell_{i}<u_{i}$ for $i=1,\dots,n$, and $f:\mathbb{R}^{n}\to\mathbb{R}$ being a twice continuously differentiable function. Specifically, let us assume that
\vspace{0.1cm}
\begin{mdframed}
\noindent\textbf{A1.} $\nabla f:\mathbb{R}^{n}\to\mathbb{R}^{n}$ is $L_{g}$-Lipschitz continuous.\\
\noindent\textbf{A2.} $f(\,\cdot\,)$ is bounded from below by $f_{\mathrm{low}}$. 
\end{mdframed}
\vspace{0.1cm}

Denoting $[n]:=\left\{1,\ldots,n\right\}$, for any $x\in\Omega$, define the sets $\mathcal{A}_{0}(x)=\{j\in [n] : x_{j}=\ell_{j}\}$, $\mathcal{A}_{1}(x)=\{j\in [n] : x_{j}=u_{j}\}$, $\mathcal{A}(x)=\mathcal{A}_{0}(x)\cup\mathcal{A}_{1}(x)$, and $\mathcal{I}(x)=\left\{1,\dots,n\right\}\setminus\mathcal{A}(x)$. The \textit{face} to which $x \in \Omega$ belongs is the set 
\[
\mathcal{F}(x)=\left\{z\in\Omega : z_{i}=\ell_{i} \text{ if } i\in\mathcal{A}_{0}(x), z_{i}=u_{i} \text{ if } i\in\mathcal{A}_{1}(x),\,\ell_{i}<z_{i}<u_{i} \text{ otherwise}\right\}. 
\]
The reduced gradient of $f$ at $x$ with respect to $\Omega$ is defined by 
\[
\nabla_{\Omega}f(x)=x-P_{\Omega}(x-\nabla f(x)),
\]
where $P_{\Omega}(\,\cdot\,)$ is the projection operator onto $\Omega$. Let $\nabla_{\Omega}^{I}f(x)\in\mathbb{R}^{n}$ be the vector defined by $\left[\nabla_{\Omega}^{I}f(x)\right]_{i} = [\nabla_{\Omega}f(x)]_{i}$, if $i\in \mathcal{I}(x)$, and $\left[\nabla_{\Omega}^{I}f(x)\right]_{i} = 0$, otherwise. Considering the enumeration $\mathcal{I}(x)=\left\{i_{1},\dots,i_{|\mathcal{I}(x)|}\right\}$, with $i_{j}<i_{j+1}$ for $j=1,\dots,|\mathcal{I}(x)|-1$ when $|\mathcal{I}(x)|\geq 2$, let us define $Q(x)=\left[e_{i_{1}}\,\,\dots\,\,e_{i_{|\mathcal{I}(x)|}}\right]\in\mathbb{R}^{n\times |\mathcal{I}(x)|}$, where $e_{i}$ denotes the $i$-th vector of the canonical basis of $\mathbb{R}^{n}$. In what follows, we will consider the function $f_{x}:\mathbb{R}^{|\mathcal{I}(x)|}\to\mathbb{R}$ given by
\begin{equation}
f_{x}(y)=f(x+Q(x)y).
\label{eq:fx}
\end{equation}
We have $\nabla f_{x}(y)=Q(x)^{T}\nabla f(x+Q(x)y)$ and $\nabla^{2}f_{x}(y)=Q(x)^{T}\nabla^{2}f(x+Q(x)y)Q(x)$,
and so 
\begin{equation}
\nabla f_{x}(0)=Q(x)^{T}\nabla f(x)\quad\text{and}\quad\nabla^{2} f_{x}(0)=Q(x)^{T}\nabla^{2} f(x)Q(x).
\label{eq:gx}
\end{equation}

\subsection{Method that employs SPG for leaving faces} \label{sec21}

Our first method (Algorithm 1) is inspired by Algorithm 4.1 in~\cite{BM}. In its $k$th iteration, we check whether the gradient norm in the current face, $\|\nabla_{\Omega}^{I} f(x^{k})\|$, is greater than or equal to a multiple of the reduced gradient norm, $\|\nabla_{\Omega} f(x^{k})\|$. If this condition holds, then $x^{k+1}$ is computed using Algorithm 2, which performs one iteration of a descent method on $f_{x^{k}}$, where the descent direction is obtained by MINRES. Otherwise, $x^{k+1}$ is computed by Algorithm 3, which applies a monotone iteration of the Spectral Projected Gradient (SPG) method~\cite{Birgin2000} to the minimization of $f$ in $\Omega$.

\begin{mdframed}
\noindent\textbf{Algorithm 1.} Active-Set Newton-MR method with SPG 
\\[0.2cm]
\noindent\textbf{Step 0.} Given $x^{0}\in\Omega$, $\epsilon>0$, $\theta \in (0,1]$, $\rho\in (0,1)$, $m\in\mathbb{N}$, $a_{1}\geq 1$, $a_2\in(0,1)$, $\lambda_{\max}^{\mathrm{spg}} \geq \lambda_{\min}^{\mathrm{spg}}>0$ be given. Set $k:=0$.
\\[0.2cm]
\noindent\textbf{Step 1.} If $\|\nabla_{\Omega}f(x^{k})\|\leq \epsilon$, STOP.
\\[0.2cm]
\noindent\textbf{Step 2.} If $\left\|\nabla_{\Omega}^{I}f(x^{k})\right\|\geq \theta \|\nabla_{\Omega}f(x^{k})\|$, compute $x^{k+1}$ by applying Algorithm~2 with parameters $x^k$, $\rho$, $m$, $a_1$, and $a_2$. Otherwise, compute $x^{k+1}$ by applying Algorithm~3 with parameters $x^k$, $\rho$, $\lambda_{\max}^{\mathrm{spg}}$, and $\lambda_{\min}^{\mathrm{spg}}$.
\\[0.2cm]
\noindent\textbf{Step 3.} Set $k:=k+1$ and go to Step 1.
\end{mdframed}

Within faces, we apply Algorithm 2, which is a descent method with Armijo line search and extrapolation. To obtain the search direction, we first compute an approximate solution $d_1^k$ to the minimum residual problem associated with the Newton system. This solution is then corrected to produce a direction $d^k$ satisfying conditions $\|d^{k}\| \leq a_{1} \|\nabla f_{x^{k}}(0)\|$ and $\langle\nabla f_{x^{k}}(0), d^{k} \rangle \leq -a_{2} \|\nabla f_{x^{k}}(0)\|^{2}$ (see Lemma 2.1).

\begin{mdframed}
\noindent\textbf{Algorithm 2.} MINRES-Based Descent Method
\\[0.2cm]
\noindent\textbf{Inputs:} $x^{k}\in\Omega$, $\rho\in (0,1)$, $m\in\mathbb{N}$, $a_{1}\geq 1$, $a_2\in(0,1)$.
\\[0.2cm]
\noindent\textbf{Step 1.} Using MINRES, compute an approximate solution $d_{1}^{k}$ of the subproblem
\[
\min_{s\in\mathbb{R}^{|\mathcal{I}(x^{k})|}}\|\nabla^{2}f_{x^{k}}(0)s+\nabla f_{x^{k}}(0)\|_{2}^{2}.
\]
Define 
\begin{equation} \label{eq:direction1}
d_{2}^{k}=\beta_{1}d_{1}^{k},
\end{equation}
where 
\begin{equation} \label{eq:beta1}
\text{if } \|d_{1}^{k}\|\leq a_{1}\|\nabla f_{x^{k}}(0)\| \text{ then } \beta_1=1 
\text{ else } \beta_1=a_{1} \|\nabla f_{x^{k}}(0)\| / \|d_{1}^{k}\|.
\end{equation}
and then set
\begin{equation} \label{eq:direction2}
d^{k}=\beta_{2}d_{2}^{k}+(1-\beta_{2})\left(-\nabla f_{x^{k}}(0)\right),
\end{equation}
where
\begin{equation} \label{eq:beta2}
\text{if } \langle \nabla f_{x^{k}}(0), d^{k}_{2} \rangle \leq -a_{2} \|\nabla f_{x^{k}}(0)\|^{2} \text{ then } \beta_2=1 \text{ else }
\beta_2 = ( 1-a_{2} ) / \left(1+\frac{\left\langle\nabla f_{x^{k}}(0),d_{2}^{k}\right\rangle}{\|\nabla f_{x^{k}}(0)\|^{2}}\right).
\end{equation}

\noindent\textbf{Step 2.1.} If $x^{k}+Q(x^{k})d^{k} \in \mathcal{F}(x^{k})$, set $\alpha_0=1$ and go to Step~3.
\\[0.2cm]
\noindent\textbf{Step 2.2.} If $f(P_{\Omega}(x^k+Q(x^{k})d^{k})) \leq f_{x^{k}}(0)$, then using at most $m$ oracle calls, find $x^{k+1} \in \partial \mathcal{F}(x^{k})$ such that $f(x^{k+1}) \leq f(P_{\Omega}(x^k+Q(x^k) d^k))$ and STOP.
\\[0.2cm]
\noindent\textbf{Step 2.3.} Compute $t_{\max}=\max\left\{t\in (0,1]\,:\,x^{k}+t Q(x^{k})d^{k}\in\Omega\right\}$. If $f_{x^k}(t_{\max} d^k) \leq f_{x^{k}}(0)$, using at most $m$ oracle calls, find $x^{k+1} \in \partial \mathcal{F}(x^{k})$ such that $f(x^{k+1}) \leq f_{x^k}(t_{\max} d^k)$ and STOP. Otherwise, set $\alpha_0 = t_{\max}$.
\\[0.2cm]
\noindent\textbf{Step 3.} Find the smallest nonnegative integer $\ell_{k}$ such that
\begin{equation} \label{eq:2.3}
f_{x^k}((0.5)^{\ell_{k}} \alpha_0 d^k) \leq f_{x^{k}}(0)+\rho (0.5)^{\ell_{k}} \alpha_0 \langle\nabla f_{x^{k}}(0),d^{k}\rangle.
\end{equation}

\noindent\textbf{Step 4.} If $\ell_k>0$, define $x^{k+1}=x^k+(0.5)^{\ell_{k}} \alpha_0 Q(x^k) d^k$. Otherwise, using at most $m$ oracle calls, find $x^{k+1} \in \overline{\mathcal{F}}(x^{k})$ such that $f(x^{k+1}) \leq f_{x^k}(\alpha_0 d^k)$.
\end{mdframed}

As a leaving-face algorithm, we use a monotone variant of the SPG method, which applies Armijo line search with the direction $v^{k} = P_{\Omega} \left( x^{k} - (1/\lambda_{k}^{\mathrm{spg}}) \nabla f(x^{k}) \right) - x^{k}$, where \(\lambda_{k}^{\mathrm{spg}}\) is the Barzilai-Borwein stepsize with safeguards.

\begin{mdframed}
\noindent\textbf{Algorithm 3.} Monotone SPG \cite{bmcomper}
\\[0.2cm]
\noindent\textbf{Inputs:} $x^{k}\in\Omega$, $\rho\in (0,1)$, $\lambda_{\max}^{\mathrm{spg}} \geq \lambda_{\min}^{\mathrm{spg}}>0$.
\\[0.2cm]
 \noindent\textbf{Step 1.} Choose $\lambda_{k}^{\mathrm{spg}} \in \left[\lambda_{\min}^{\mathrm{spg}}, \lambda_{\max}^{\mathrm{spg}}\right]$ and compute the solution $v^{k}$ to 
\[
\mathop{\min_{v\in\Omega\setminus\{x^{k}\}}} \nabla f(x^{k})^{T}v + (\lambda_{k}^{\mathrm{spg}}/2) \|v\|^{2}.
\]
\noindent\textbf{Step 2.} Find the smallest nonnegative integer $j_{k}$ such that
\[
f(x^{k}+(0.5)^{j_{k}} v^{k})\leq f(x^{k})+\rho (0.5)^{j_{k}}\langle\nabla f(x^{k}),v^{k}\rangle.
\]
\noindent\textbf{Step 3.} Set $x^{k+1}=x^{k}+(0.5)^{j_{k}}v^{k}$.
\end{mdframed}

Let us begin our analysis by focusing on the properties of Algorithm 2. Our first lemma establishes the descent properties of the search direction $d^{k}$.

\begin{lemma}
Let $d^{k}$ be defined by (\ref{eq:direction2}) and (\ref{eq:beta2}). Then
\begin{equation}
\|d^{k}\|\leq a_{1}\|\nabla f_{x^{k}}(0)\|\quad\text{and}\quad\langle\nabla f_{x^{k}}(0),d^{k}\rangle\leq -a_{2}\|\nabla f_{x^{k}}(0)\|^{2}.
\label{eq:2.2}
\end{equation}
\label{lem:extra1}
\end{lemma}

\begin{proof}
By (\ref{eq:direction1}) and (\ref{eq:beta1}), we have
\begin{equation} \label{eq:luke}
\|d_{2}^{k}\|\leq a_{1}\|\nabla f_{x^{k}}(0)\|.
\end{equation}
Let us now analyze $d^{k}$ defined by (\ref{eq:direction2}) and (\ref{eq:beta2}). If $\left\langle\nabla f_{x^{k}}(0),d_{2}^{k}\right\rangle\leq-a_{2}\|\nabla f_{x^{k}}(0)\|^{2}$ then by (\ref{eq:beta2}) we have $\beta_{2}=1$. Thus, in view of (\ref{eq:luke}), we obtain
\[
\|d^{k}\|=\|d_{2}^{k}\|\leq a_{1}\|\nabla f_{x^{k}}(0)\|\;\text{and}\;\left\langle\nabla f_{x^{k}}(0),d^{k}\right\rangle=\left\langle\nabla f_{x^{k}}(0),d_{2}^{k}\right\rangle\leq -a_{2}\|\nabla f_{x^{k}}(0)\|^{2},
\]
that is (\ref{eq:2.2}) holds. Suppose now that $\left\langle\nabla f_{x^{k}}(0),d_{2}^{k}\right\rangle>-a_{2}\|\nabla f_{x^{k}}(0)\|^{2}$. Then we have
\[
1 + \langle \nabla f_{x^{k}}(0),d_{2}^{k} \rangle / \| \nabla f_{x^{k}}(0) \|^{2} > 1 - a_{2} > 0,
\]
and so, by (\ref{eq:beta2}), 
\[
0<\beta_{2}=(1-a_{2}) / \left(1+\langle\nabla f_{x^{k}}(0),d_{2}^{k}\rangle / \|\nabla f_{x^{k}}(0)\|^{2}\right) <1.
\]
Consequently, it follows from (\ref{eq:direction2}) and (\ref{eq:luke}) that
\[
\|d^{k}\|\leq\beta_{2}\|d_{2}^{k}\|+(1-\beta_{2})\|\nabla f_{x^{k}}(0)\|\leq a_{1}\|\nabla f_{x^{k}}(0)\|.
\]
In addition, by (\ref{eq:direction2}) and (\ref{eq:beta2}), defining $\nu_k = 1+\langle\nabla f_{x^{k}}(0),d_{2}^{k}\rangle / \|\nabla f_{x^{k}}(0)\|^{2}$ and $\xi_k = \|\nabla f_{x^{k}}(0)\|^{2}+\left\langle\nabla f_{x^{k}}(0),d_{2}^{k}\right\rangle$, we also have
\begin{eqnarray*}
\langle\nabla f_{x^{k}}(0),d^{k}\rangle
&=& \beta_{2}\langle\nabla f_{x^{k}}(0),d_{2}^{k}\rangle-(1-\beta_{2})\|\nabla f_{x^{k}}(0)\|^{2}\\
&=&\left((1-a_{2})\left\langle\nabla f_{x^{k}}(0),d_{2}^{k}\right\rangle-\left(\nu_k-(1-a_{2})\right)\|\nabla f_{x^{k}}(0)\|^{2}\right) / \nu_k\\
&=&\left(-a_{2}\left\langle\nabla f_{x^{k}}(0),d_{2}^{k}\right\rangle-a_{2}\|\nabla f_{x^{k}}(0)\|^{2}\right) / \nu_k\\
&=&\left((-a_{2} \xi_k ) / \xi_k \right) \|\nabla f_{x^{k}}(0)\|^{2} = -a_{2}\|\nabla f_{x^{k}}(0)\|^{2}.
\end{eqnarray*}
Therefore, (\ref{eq:2.2}) also holds in this case.
\end{proof}

Combining Lemma 2.1 and the Lipschitz continuity of $\nabla f$ (assumption A1), the next lemma provides a positive lower bound for stepsizes that do not satisfy the Armijo condition.

\begin{lemma}
\label{lem:2.0}
Suppose that A1 holds. Given $\alpha>0$, if 
\begin{equation}
f_{x^{k}}(\alpha d^{k})>f_{x^{k}}(0)+\rho\alpha\langle\nabla f_{x^{k}}(0),d^{k}\rangle
\label{eq:2.4extra}
\end{equation}
then $\alpha > 2(1-\rho)a_{2} / (L_{g}a_{1}^{2})$.
\end{lemma}

\begin{proof}
Using the definition of $f_{x^{k}}$ in (\ref{eq:fx}), we see that inequality (\ref{eq:2.4extra}) is equivalent to
\begin{equation}
\rho \, \alpha\langle\nabla f_{x^{k}}(0),d^{k}\rangle+f(x^{k})<f\left(x^{k}+\alpha Q(x^{k})d^{k}\right).
\label{eq:2.8}
\end{equation}
Notice that $Q(x^{k})^{T}Q(x^{k})=Id$. Thus, by assumption A1 we also have
\begin{equation} \label{eq:2.9}
f(x^{k}+\alpha Q(x^{k})d^{k}) \leq
f(x^{k})+\alpha[\langle\nabla f_{x^{k}}(0),d^{k}\rangle+(L_{g}/2)\alpha\|d^{k}\|^{2}].
\end{equation}

From (\ref{eq:2.8}) and (\ref{eq:2.9}), it follows that $\rho \langle \nabla f_{x^{k}}(0), d^{k} \rangle < \langle \nabla f_{x^{k}}(0), d^{k} \rangle + (L_{g}/2) \alpha \|d^{k}\|^{2}$. Therefore, by Lemma \ref{lem:extra1}, we conclude that
\[
\alpha > \left( 2(1-\rho) / L_{g} \right) \left(-\langle \nabla f_{x^{k}}(0), d^{k} \rangle / \|d^{k}\|^{2} \right) \geq 2(1-\rho)a_{2} / (L_{g}a_{1}^{2}).
\]
\end{proof}

In view of Lemmas 2.1 and 2.2, if $x^{k+1}$ is computed by Algorithm 2, then the objective function decreases by at least a multiple $\|\nabla_{\Omega}^{I}f(x^{k})\|^{2}$.

\begin{lemma}
\label{lem:2.1}
Suppose that A1 holds. Then, whenever $x^{k+1}$ is computed by Step 4 of Algorithm 2, we have
\begin{equation}
f(x^{k})-f(x^{k+1})\geq\rho\min\left\{a_{2},(1-\rho)a_{2}^{2}/(L_{g}a_{1}^{2})\right\}\left\|\nabla_{\Omega}^{I}f(x^{k})\right\|^{2}.
\label{eq:2.5}
\end{equation}
Moreover, the number of evaluations of $f$ necessary to guarantee the fulfillment of~(\ref{eq:2.3}) is bounded from above by $\left|\log_{2}\left(\min\left\{1/2,(1-\rho)a_{2}/(2L_{g}a_{1}^{2})\right\}\right)\right|$.
\end{lemma}

\begin{proof}
First, let us show that 
\begin{equation}
(0.5)^{\ell_{k}}\alpha_{0}\geq\min\left\{\alpha_{0}, (1-\rho)a_{2} / (L_{g}a_{1}^{2})\right\}.
\label{eq:2.7extra}
\end{equation}
If $\ell_{k}=0$, then (\ref{eq:2.7extra}) is clearly true. Thus, let us assume that $\ell_{k}>0$. In this case, by the definition of $\ell_{k}$, it follows that
\[
f_{x^{k}}\left((0.5)^{\ell_{k}-1}\alpha_{0}d^{k}\right)>f_{x^{k}}(0)+\rho (0.5)^{\ell_{k}-1}\alpha_{0}\langle\nabla f_{x^{k}}(0),d^{k}\rangle.
\]
Thus, by Lemma \ref{lem:extra1}, we have
$(0.5)^{\ell_{k}-1}\alpha_{0}>2(1-\rho)a_{2}/(L_{g}a_{1}^{2})$,
which implies that (\ref{eq:2.7extra}) also holds when $\ell_{k}>0$. Now, let us refine the lower bound in (\ref{eq:2.7extra}). In view of Steps 2.1 and 2.3 of Algorithm 2, we have $\alpha_{0}=1$ or $\alpha_{0}=t_{\max}\leq 1$. The latter occurs only if $f_{x^{k}}(t_{\max}d^{k})>f_{x^{k}}(0)$. In particular, this means that inequality (\ref{eq:2.4extra}) holds for $\alpha=t_{\max}$. Consequently, by Lemma \ref{lem:2.0}, in this case we must have $\alpha_{0}=t_{\max}>2(1-\rho)a_{2}/(L_{g}a_{1}^{2})$. Therefore, in any case, we have
\begin{equation}
\alpha_{0}\geq\min\left\{1,2(1-\rho)a_{2}/(L_{g}a_{1}^{2})\right\}.
\label{eq:2.8extra}
\end{equation}
Combining (\ref{eq:2.7extra}) and (\ref{eq:2.8extra}), we obtain
\begin{equation}
(0.5)^{\ell_{k}}\alpha_{0}\geq\min\left\{1,(1-\rho)a_{2}/(L_{g}a_{1}^{2})\right\}.
\label{eq:2.7}
\end{equation}
Now, from Steps 3 and 4 of Algorithm 2, we obtain
\begin{eqnarray*}
f(x^{k})-f(x^{k+1})
&\geq &\rho (0.5)^{\ell_{k}}\alpha_{0}\left(-\langle\nabla f_{x^{k}}(0),d^{k}\rangle\right)\geq \rho (0.5)^{\ell_{k}}\alpha_{0}a_{2}\|\nabla f_{x^{k}}(0)\|^{2}\\
& = & \rho (0.5)^{\ell_{k}}\alpha_{0}a_{2}\left\|Q(x^{k})^{T}\nabla f(x^{k})\right\|^{2}\geq \rho (0.5)^{\ell_{k}}\alpha_{0}a_{2}\left\|\nabla_{\Omega}^{I}f(x^{k})\right\|^{2}\\
&\geq &\rho\min\left\{1,(1-\rho)a_{2}/(L_{g}a_{1}^{2})\right\}a_{2}\left\|\nabla_{\Omega}^{I}f(x^{k})\right\|^{2},
\end{eqnarray*}
that is, (\ref{eq:2.5}) is true. Finally, notice that the number of evaluations of $f$ performed at Step 3 of Algorithm 2 is equal to $\ell_{k}+1$. From (\ref{eq:2.7}), we have
\[
(0.5)^{\ell_{k}+1}\geq (0.5)^{\ell_{k}+1}\alpha_{0}\geq\min\left\{1/2,(1-\rho)a_{2}/(2L_{g}a_{1}^{2})\right\}.
\]
Then, taking the logarithm on both sides, it follows that
\[
\ell_{k}+1 \leq \left|\log_{2}\left( \min\left\{1/2,(1-\rho)a_{2}/(2L_{g}a_{1}^{2})\right\} \right)\right|.
\]
\end{proof}

If $x^{k+1}$ is computed by Algorithm 3, then the objective function decreases by at least a constant multiple of $\|\nabla_{\Omega} f(x^{k})\|^{2}$.

\begin{lemma}
\label{lem:2.2} (Theorem 4.2 in \cite{bmcomper}) Suppose that A1 holds. Then, whenever $x^{k+1}$ is computed by Algorithm 3, we have
\[
f(x^{k})-f(x^{k+1})\geq \rho(1-2\rho) / (8L_{g}) \min\left\{\lambda_{\min}^{\mathrm{spg}}, \lambda_{\min}^{\mathrm{spg}} / \lambda_{\max}^{\mathrm{spg}} \right\}^{2}\left\|\nabla_{\Omega} f(x^{k})\right\|^{2}.
\]
Moreover, the number of evaluations of $f$ necessary to guarantee the fulfillment of (\ref{eq:2.3}) is bounded from above by $\left|\log_{2}\left(\min\left\{1,(1-2\rho)\lambda_{\min}^{\mathrm{spg}}/(4L_{g})\right\}\right)\right|+1$.
\end{lemma}

Let $\left\{x^{k}\right\}_{k\geq 0}$ be a sequence generated by Algorithm 1. Denote by
\begin{equation}
    T(\epsilon)=\inf\left\{k\in\mathbb{N}\,:\,\|\nabla_{\Omega} f(x^{k})\|\leq\epsilon\right\},
    \label{eq:hitting}
\end{equation}
the first hitting time to the set of $\epsilon$-approximate stationary points of $f$ with respect to $\Omega$, and consider the sets 
\small
\begin{eqnarray*}
\mathcal{S}_{T(\epsilon)-1} &=& \{ k\in\left\{0,\dots,T(\epsilon)-1\right\}: \text{$x^{k+1}$ is computed in Step 4 of Alg. 2 or by Alg. 3}\}\\
\mathcal{U}_{T(\epsilon)-1} &=& \{k\in\left\{0,\dots,T(\epsilon)-1\right\} : \text{$x^{k+1}$ is computed in Step 2.2 or Step 2.3 of Alg. 2}\}
\end{eqnarray*} 
\normalsize
The next theorem establishes that $T(\epsilon) \leq \mathcal{O}\left(n\epsilon^{-2}\right)$, that is, Algorithm~1 needs no more than $\mathcal{O}(n\epsilon^{-2})$ iterations to find $\epsilon$-approximate stationary points.

\begin{theorem} \label{thm:2.1}
Suppose that assumptions A1 and A2 hold, and let $\left\{x^{k}\right\}_{k=0}^{T(\epsilon)}$ be generated by Algorithm 1. Then
\begin{equation}
T(\epsilon)\leq 1+(n+1) (f(x^{0})-f_{\mathrm{low}}) \kappa_{1}^{-1} \epsilon^{-2},
\label{eq:2.12}
\end{equation}
where
\begin{equation}
\kappa_1 = \min\left\{\rho a_{2}\theta^{2},(\rho(1-\rho)/L_{g})(\theta a_{2}/a_{1})^{2}, \rho(1-2\rho)/(8L_{g})\min\{\lambda_{\min}^{\mathrm{spg}},\lambda_{\min}^{\mathrm{spg}}/\lambda_{\max}^{\mathrm{spg}}\right\}^{2}\}.
\label{eq:2.13}
\end{equation}
\end{theorem}
\begin{proof}
If $T(\epsilon)\leq 1$, then (\ref{eq:2.12}) is clearly true. Thus, suppose that $T(\epsilon)\geq 2$ and let $k\in\left\{0,\dots,T(\epsilon)-1\right\}$. By the definition of $T(\epsilon)$ we have
\begin{equation} \label{eq:2.14}
\|\nabla_{\Omega}f(x^{k})\|>\epsilon.
\end{equation}
If 
\begin{equation} \label{eq:2.15}
\left\|\nabla_{\Omega}^{I}f(x^{k})\right\|\geq \theta\|\nabla_{\Omega}f(x^{k})\|
\end{equation}
then, by Step 2 of Algorithm 1, the next iterate $x^{k+1}$ is computed by Algorithm 2. Specifically, if $x^{k+1}$ is obtained from Step 4 of Algorithm 2, it follows from Lemma~\ref{lem:2.1}, (\ref{eq:2.15}), (\ref{eq:2.14}) and (\ref{eq:2.13}) that
\begin{eqnarray}
f(x^{k})-f(x^{k+1})
&\geq &\rho\min\left\{a_{2},(1-\rho)a_{2}^{2}/(L_{g}a_{1}^{2})\right\}\theta^{2}\|\nabla_{\Omega}f(x^{k})\|^{2}
\geq \kappa_{1}\epsilon^{2}.
\label{eq:2.16}
\end{eqnarray}
On the other hand, if $x^{k+1}$ is computed by Algorithm 3, it follows from Lemma \ref{lem:2.1}, (\ref{eq:2.15}) and (\ref{eq:2.13}) that
\begin{equation} \label{eq:2.17}
f(x^{k})-f(x^{k+1})\geq (\rho(1-2\rho)/(8L_{g}))\min\left\{\lambda_{\min}^{\mathrm{spg}},\lambda_{\min}^{\mathrm{spg}}/\lambda_{\max}^{\mathrm{spg}}\right\}^{2}\epsilon^{2}\geq \kappa_{1}\epsilon^{2}.
\end{equation}
In view of (\ref{eq:2.16}) and (\ref{eq:2.17}), for any $k\in\mathcal{S}_{T(\epsilon)-1}$ we have $f(x^{k})-f(x^{k+1})\geq \kappa_{1}\epsilon^{2}$.
In addition, for any $k\in\mathcal{U}_{T(\epsilon)-1}$, we have $f(x^{k+1})\leq f(x^{k})$. Thus, by assumption A1,
\begin{eqnarray*}
f(x^{0})-f_{\mathrm{low}}
&\geq&\sum_{k=0}^{T(\epsilon)-1} f(x^{k})-f(x^{k+1})\geq \sum_{k\in\mathcal{S}_{T(\epsilon)-1}}f(x^{k})-f(x^{k+1})\geq |\mathcal{S}_{T(\epsilon)-1}|\kappa_{1}\epsilon^{2},
\end{eqnarray*}
which implies that
\begin{equation} \label{eq:2.18}
|\mathcal{S}_{T(\epsilon)-1}|\leq(f(x^{0})-f_{\mathrm{low}}) \kappa_{1}^{-1} \epsilon^{-2}.
\end{equation}
Notice that if $k\in\mathcal{U}_{T(\epsilon)-1}$, then $x^{k+1}\in\partial \mathcal{F}(x^{k+1})$, and so $|\mathcal{I}(x^{k+1})|\leq |\mathcal{I}(x^{k})|-1$. Consequently, there can be at most $n$ consecutive iterations of Algorithm 1 with $x^{k+1}$ being computed by Step~2.3 of Algorithm~2. In the worst case, each iteration in $\mathcal{S}_{T(\epsilon)-1}$ would be followed by $n$ consecutive iterations in $\mathcal{U}_{T(\epsilon)-1}$. Therefore,
\begin{equation} \label{eq:2.19}
|\mathcal{U}_{T(\epsilon)-1}|\leq n|\mathcal{S}_{T(\epsilon)-1}|.
\end{equation}
Finally, combining (\ref{eq:2.18}) and (\ref{eq:2.19}), we conclude that
\[
T(\epsilon) = |\mathcal{S}_{T(\epsilon)-1}| + |\mathcal{U}_{T(\epsilon)-1}| \leq (1+n) (f(x^{0})-f_{\mathrm{low}}) \kappa_{1}^{-1} \epsilon^{-2},
\]
which means that (\ref{eq:2.12}) also holds when $T(\epsilon)\geq 2$.
\end{proof}

\begin{remark}
If $k\in\mathcal{S}_{T(\epsilon)-1}$ then it follows from Lemmas \ref{lem:2.1} and \ref{lem:2.2} that the number of evaluations of $f$ at the $k$-th iteration is bounded from above by 
\[
(m+3)+\max\left\{\left|\log_{2}\left(\min\left\{\frac{1}{2},\frac{(1-\rho)a_{2}}{2L_{g}a_{1}^{2}}\right\}\right)\right|,\left|\log_{2}\left(\min\left\{1,\frac{(1-2\rho)\lambda_{\min}^{\mathrm{spg}}}{4L_{g}}\right\}\right)\right|\right\}.
\]
Additionally, there will be one gradient computation and one Hessian computation. On the other hand, if $k\in\mathcal{U}_{T(\epsilon)-1}$, then at the $k$-th iteration there will be one gradient computation, one Hessian computation and, at most, $m+1$ function evaluations. In summary, each iteration of Algorithm 1 requires at most 
\[
(m+5)+\max\left\{\left|\log_{2}\left(\min\left\{\frac{1}{2},\frac{(1-\rho)a_{2}}{2L_{g}a_{1}^{2}}\right\}\right)\right|,\left|\log_{2}\left(\min\left\{1,\frac{(1-2\rho)\lambda_{\min}^{\mathrm{spg}}}{4L_{g}}\right\}\right)\right|\right\}
\]
calls to the oracle. Then, in view of Theorem \ref{thm:2.1}, Algorithm 1 needs no more than $\mathcal{O}\left(n\epsilon^{-2}\right)$ calls to the oracle to find $x^{k}$ such that $\|\nabla_{\Omega} f(x^{k})\|\leq\epsilon$.
\end{remark}

\subsection{Method that employs cubic regularization for leaving faces} \label{sec22}

In what follows, we consider a method (Algorithm 4) that uses an adaptation of the Newton MINRES method \cite{Liu2023} within the faces, while it employes the Cubic Regularization of the Newton's method \cite{BM} to leave the faces. Regarding the MINRES method for approximately solving $\min_{s\in\mathbb{R}^{p}}\|Hs+g\|_{2}^{2}$, we use Algorithm 1 from \cite{Liu2023}, which is called by $\left[s,\texttt{D}_{\texttt{type}}\right] = \texttt{MINRES}\left(H,g,\eta\right)$. Denote the residual by $r=-(Hs+g)$ and consider
\begin{eqnarray} 
&\langle g,s\rangle+\langle Hs,s\rangle\leq 0,& \label{eq:minres1} \\
&\|Hr\|\leq\eta\|Hs\|,& \label{eq:minres3} \\
&\langle g,r\rangle=-\|r\|^{2},& \label{eq:minres4} \\
&\langle Hr,r\rangle<0.& \label{eq:minres5}
\end{eqnarray}
In view of Lemmas 11--13 in \cite{Liu2023}, if the number of iterations performed by MINRES does not exceed the grade of $g$ with respect to $H$ \cite[Def.~1]{Liu2023}, then the outcome can take only one of the following two forms:
\vspace{0.2cm}
\begin{itemize}
  \item $\texttt{D}_{\texttt{type}}=\texttt{`SOL'}$, meaning that 
  $\langle Hs,s\rangle>0$, $\langle Hr,r\rangle>0$, \eqref{eq:minres1}, and~\eqref{eq:minres3} hold;
  \item $\texttt{D}_{\texttt{type}}=\texttt{`NPC'}$, meaning that 
  \eqref{eq:minres4} and~\eqref{eq:minres5} hold.
\end{itemize}
\vspace{0.2cm}
\noindent As in \cite{Liu2023}, we shall assume throughout the remainder of our analysis that this dichotomy always applies, that is, whenever MINRES is applied, the number of iterations does not exceed the grade of the corresponding linear least squares problem.

Algorithm 4 follows a structure similar to that of Algorithm 1, but with some important differences. In addition to using different methods within faces and for leaving faces, a key distinction is the switching mechanism between these methods. Specifically, whenever Newton-MR neither produces a functional decrease of order $\mathcal{O}(\epsilon^{3/2})$ nor returns a point on the boundary of the face that results in a simple decrease of the objective function, we switch to the Cubic Regularization method for the next iterate, which is guaranteed to produce a functional decrease of $\mathcal{O}(\epsilon^{3/2})$.

\begin{mdframed}
\noindent\textbf{Algorithm 4.} Active-Set Newton-MR method with Cubic Regularization
\\[0.1cm]
\noindent\textbf{Step 0.} Given $x^{0}\in\Omega$, $\epsilon>0$, $\theta,\eta\in (0,1]$, $\tau \in (0,1]$, $\rho\in (0,1/2)$, $m\in\mathbb{N}$, $a_{1}\geq 1>a_{2}>0$, and $M,\alpha,\gamma>0$ be given. Set $\eta_{0} \geq \eta$, $M_{0}=M$, $\sigma_{0}=0$ and $k:=0$.
\\[0.1cm]
\noindent\textbf{Step 1.} If $\|\nabla_{\Omega}f(x^{k})\|\leq \epsilon$, STOP.
\\[0.1cm]
\noindent\textbf{Step 2.} If $\sigma_{k}\in\left\{0,2\right\}$ and $\left\|\nabla_{\Omega}^{I}f(x^{k})\right\|\geq \theta\|\nabla_{\Omega}f(x^{k})\|$, call Algorithm 5 with parameters $x^{k}$, $\rho$, $\eta_{k}$, $\eta$, $\tau$, $m$, $a_1$, and $a_2$ to compute $x^{k+1}$, $\sigma_{k+1}$, and $\eta_{k+1}$ and set $M_{k+1}=M_{k}$. Otherwise, call Algorithm 6 with parameters $x^{k}$, $M_{k}$, $M$, $\alpha$, and $\gamma$ to compute $x^{k+1}$ and $M_{k+1}$ and set $\sigma_{k+1}=0$ and $\eta_{k+1}=\eta_{k}$.
\\[0.1cm]
\noindent\textbf{Step 3.} Set $k:=k+1$ and go to Step 1.
\end{mdframed}

In what follows, we describe in detail the Newton-MR algorithm.

\begin{mdframed}
\noindent\textbf{Algorithm 5.} Newton-MR
\\[0.2cm]
\noindent\textbf{Inputs:} $x^{k}\in\Omega$, $\rho\in (0,1/2)$, $\eta_{k}\in [\eta,1]$, $\eta \in (0,1]$, $\tau \in (0,1]$, $m\in\mathbb{N}$, $a_{1}\geq 1>a_{2}>0$. 
\\[0.2cm]
\noindent\textbf{Step 1.} Call Minimal Residual method as $[s^{k},\texttt{D}_{\texttt{type}}^{k}] = \texttt{MINRES}\left(\nabla^{2}f_{x^{k}}(0),\nabla f_{x^{k}}(0),\eta_{k}\right)$ and set $r^{k}=-(\nabla^{2}f_{x^{k}}(0)s^{k}+\nabla f_{x^{k}}(0))$.
Define
\begin{equation} \label{eq:direction}
d_{1}^{k}=\left\{\begin{array}{ll} s^{k},&\text{if $\texttt{D}_{\texttt{type}}^{k}=\texttt{`SOL'}$},\\
r^{k},&\text{if $\texttt{D}_{\texttt{type}}^{k}=\texttt{`NPC'}$ }.
\end{array}
\right.
\end{equation}
Set 
\begin{equation} \label{eq:flag}
\texttt{Flag}=
\left\{
\begin{array}{ll} 
1, & \text{if } \texttt{D}_{\texttt{type}}^{k}=\texttt{`SOL'} 
\text{ and } \min\left\{\frac{\langle\nabla^{2}f_{x^{k}}(0)d_{1}^{k},d_{1}^{k}\rangle}{\|d_{1}^{k}\|^{2}},\frac{\langle\nabla^{2}f_{x^{k}}(0)r^{k},r^{k}\rangle}{\|r^{k}\|^{2}}\right\}<\eta_k,\\
1, & \text{if } \texttt{D}_{\texttt{type}}^{k}=\texttt{`NPC'} 
\text{ and } \|\nabla^{2}f_{x^{k}}(0)r^{k}\|\leq\eta_k\|\nabla^{2}f_{x^{k}}(0)s^{k}\|,\\
0, & \text{otherwise}.
\end{array}
\right.
\end{equation}
and define
\begin{equation} \label{eq:eta}
\eta_{k+1}=\left\{\begin{array}{ll} \eta_{k},&\text{if $\texttt{Flag}=0$},\\
\max\left\{\tau \  \eta_{k},\eta\right\},&\text{if $\texttt{Flag}=1$}.
\end{array}
\right.
\end{equation}
If $\texttt{Flag}=0$, define $d^{k}=d_{1}^{k}$. Otherwise, define $d_{2}^{k}=\beta_{1}d_{1}^{k}$, where $\beta_1$ is given by~\eqref{eq:beta1},
and then set
$d^{k}=\beta_{2}d_{2}^{k}+(1-\beta_{2})\left(-\nabla f_{x^{k}}(0)\right)$,
where $\beta_2$ is given by~\eqref{eq:beta2}.
\\[0.2cm]

\noindent\textbf{Step 2.1.} If $x^{k}+Q(x^{k})d^{k} \in \mathcal{F}(x^{k})$, set $\alpha_0=1$ and go to Step~3.
\\[0.2cm]
\noindent\textbf{Step 2.2.} If $f(P_{\Omega}(x^k+Q(x^{k})d^{k})) \leq f_{x^{k}}(0)$, using at most $m$ oracle calls, find $x^{k+1} \in \partial \mathcal{F}(x^{k})$ such that $f(x^{k+1}) \leq f(P_{\Omega}(x^k+Q(x^k) d^k))$, set $\sigma_{k+1}=2$ and STOP.
\\[0.2cm]
\noindent\textbf{Step 2.3.} Compute $t_{\max}=\max\left\{t\in (0,1]\,:\,x^{k}+t Q(x^{k})d^{k}\in\Omega\right\}$. If $f_{x^k}(t_{\max} d^k) \leq f_{x^{k}}(0)$, using at most $m$ oracle calls, find $x^{k+1} \in \partial \mathcal{F}(x^{k})$ such that $f(x^{k+1}) \leq f_{x^k}(t_{\max} d^k)$, set $\sigma_{k+1}=2$ and STOP. Otherwise, set $\alpha_0 = t_{\max}$.
\\[0.2cm]
\noindent\textbf{Step 3.} Find the smallest nonnegative integer $\ell_{k}$ such that
\begin{equation} \label{eq:2.21}
f_{x^k}((0.5)^{\ell_{k}} \alpha_0 d^k) \leq f_{x^{k}}(0)+\rho (0.5)^{\ell_{k}} \alpha_0 \langle\nabla f_{x^{k}}(0),d^{k}\rangle.
\end{equation}

\noindent\textbf{Step 4.} If $\ell_k>0$, define $x^{k+1}=x^k+(0.5)^{\ell_{k}} \alpha_0 Q(x^k) d^k$, set $\sigma_{k+1}=\texttt{Flag}$ and STOP.
\\[0.2cm]

\noindent\textbf{Step 5.} If $\texttt{Flag}=1$, using at most $m$ oracle calls, find $x^{k+1} \in \overline{\mathcal{F}}(x^k)$ such that $f(x^{k+1}) \leq f_{x^k}(\alpha_0 d^k)$. If $x^{k+1} \in \partial \mathcal{F}(x^k)$, set $\sigma_{k+1}=2$. Otherwise, set $\sigma_{k+1}=1$. Then STOP.
\\[0.2cm]

\noindent\textbf{Step 6.} Find the smallest nonnegative integer $j_{k}$ such that
\begin{equation} \label{eq:extrapolation}
x^{k}+2^{j}\alpha_{0}Q(x^{k})d^{k}\in\Omega\,\text{and}\, f_{x^{k}}(2^{j}\alpha_{0}d^{k})\leq f_{x^{k}}(0)+\rho 2^{j} \alpha_{0}\langle\nabla f_{x^{k}}(0),d^{k}\rangle
\end{equation}
holds with $j=j_k$ and~\eqref{eq:extrapolation} does not hold with $j=j_k+1$.
\\[0.2cm]

\noindent\textbf{Step 6.1.} If $x^{k}+2^{j_{k}+1}\alpha_{0}Q(x^{k})d^{k}\in\Omega$, define $x_{k+1}=x_{k}+2^{j_{k}}\alpha_{0}Q(x^{k})d^{k}$, $\sigma_{k+1}=0$ and STOP.
\\[0.2cm]

\noindent\textbf{Step 6.2.} Compute $t_{\max}=\max\left\{t\in (0,1]\,:\,x^{k}+t Q(x^{k})d^{k}\in\Omega\right\}$. If $f_{x^{k}}(t_{\max} d^{k}) > f_{x^{k}}(0)$, define $x_{k+1}=x_{k}+2^{j_{k}}\alpha_{0}Q(x^{k})d^{k}$, $\sigma_{k+1}=0$ and STOP.
\\[0.2cm]

\noindent\textbf{Step 6.3.} Using at most $m$ oracle calls, find $x^{k+1} \in \partial \mathcal{F}(x^k)$ with $f(x^{k+1}) \leq f_{x^k}(t_{\max} d^k)$, and set $\sigma_{k+1}=2$.
\end{mdframed}

In Step 1 of Algorithm 5, the variable $\texttt{Flag}$ characterizes the quality of the search direction $d_{1}^{k}$ obtained via MINRES. As it will be established in Lemmas 2.9, 2.12 and 2.13, when $\texttt{Flag}=0$, we can obtain a functional decrease of $\mathcal{O}(\epsilon^{3/2})$ along the search direction $d_{1}^{k}$, which is then chosen as the search direction $d^{k}$. On the other hand, when $\texttt{Flag}=1$, the descent properties of $d_{1}^{k}$ are less clear. In this case, we modify $d_{1}^{k}$ by the same procedure used in Algorithm 2, which produces a direction $d^{k}$ along which a functional decrease of $\mathcal{O}(\epsilon^{2})$ is guaranteed (Lemma 2.9). In Step 2, we check whether a unity step along \( d^{k} \) results in a point that remains within the face \( \mathcal{F}(x^{k}) \). If it does, we proceed to Step 3, where a suitable stepsize is determined using an Armijo line-search starting from \( \alpha_0 = 1 \). If not, we attempt to find a point on the boundary of the face that results in at least a simple decrease in the objective function, in which case we set $\sigma_{k+1}=2$. If this fails, we move to Step 3 and begin the Armijo line-search with the stepsize \( \alpha_0 = t_{\max} \), corresponding to the point along \( d^{k} \) that lies on the boundary of \( \mathcal{F}(x^{k}) \). If the Armijo condition is not satisfied with the stepsize \( \alpha_0 \), we define \( x^{k+1} \) as the resulting point and set \( \sigma_{k+1} = \texttt{Flag} \), stopping at Step 4. In contrast, if the Armijo condition is satisfied with \( \alpha_0 \), we attempt to obtain a larger functional decrease using extrapolation\footnote{A precise way to implement the extrapolation procedure is described in Section 3.}. The result is a point \( x^{k+1} \) that either \textbf{(a)} belongs to the boundary of the current face and produces at least a simple decrease in the objective function, where we set \( \sigma_{k+1} = 2 \) (Steps 5 and 6.3); or \textbf{(b)} produces a functional decrease of \( \mathcal{O}(\epsilon^{3/2}) \), in which case we set \( \sigma_{k+1} = 0 \) (Steps 6.1 and 6.2). Finally, it is worth noticing that when $\Omega=\mathbb{R}^{n}$ and $f(\,\cdot\,)$ is $\mu$-strongly convex, Algorithm 5 with $\eta_{k}=\eta=\mu$ essentially reduces to one iteration of Algorithm 4 in \cite{Liu2023}, as we will have $\texttt{Flag}=0$ and Steps 2.2, 2.3, 5, 6.2, and 6.3 will be unnecessary.

Let us now present the Cubic Regularization Method for bound-constrained minimization proposed in \cite{BM}. For that, given $x\in\mathbb{R}^{n}$, we define $T_{2}(x,\,\cdot\,):\mathbb{R}^{n}\to\mathbb{R}$ by $T_{2}(x,s)=\langle\nabla f(x),s\rangle+\frac{1}{2}\langle\nabla^{2}f(x)s,s\rangle$.

\begin{mdframed}
\noindent\textbf{Algorithm 6.} One Iteration of the Cubic Regularization Method \cite{BM}
\\[0.2cm]
\noindent\textbf{Inputs:} $x^{k}\in\Omega$, $M_{k}\geq M>0$, $\alpha,\gamma>0$ be given. Set $\ell:=0$.
\\[0.2cm]
\noindent\textbf{Step 1.} Compute an approximate solution $s^{\ell}\in\mathbb{R}^{n}$ of the constrained cubic subproblem 
\[\text{Minimize $T_{2}(x,s) + (2^{\ell}M_{k}) \|s\|^{3}$ subject to $x^{k}+s\in\Omega$}
\]
such that 
\[T_{2}(x^{k},s^{\ell})+(2^{\ell}M_{k})\|s^{\ell}\|^{3}\leq 0\,\,\text{and}\,\,\|\nabla_{\Omega}[T_{p}(x^{k},x-x^{k})+(2^{\ell}M_{k})\|x-x^{k}\|^{3}]|_{x=x^{k}+s^{\ell}}\|\leq\gamma\|s^{\ell}\|^{2}.
\]
\\[0.2cm]
\noindent\textbf{Step 2.} If 
\begin{equation} \label{eq:gg_extra1}
f(x^{k})-f(x^{k}+s^{\ell})\geq\alpha\|s^{\ell}\|^{3},
\end{equation}
set $\ell_{k}=\ell$, and go to Step 3. Otherwise, set $\ell:=\ell+1$ and go to Step 1. 
\\[0.2cm]
\noindent\textbf{Step 3.} Define $x^{k+1}=x^{k}+s^{\ell_{k}}$ and $M_{k+1}=\max\left\{2^{\ell_{k}-1}M_{k},M\right\}$.
\end{mdframed}

\noindent In addition to A1 and A2, we will consider the following assumption:
\vspace{0.1cm}
\begin{mdframed}
\noindent\textbf{A3.} $\nabla^{2}f:\mathbb{R}^{n}\to\mathbb{R}^{n\times n}$ is $L_{H}$-Lipschitz continuous.
\end{mdframed}
\vspace{0.1cm}

We begin our analysis by examining the properties of Algorithm 5. The next two lemmas provide positive lower bounds for stepsizes that do not satisfy the Armijo condition.

\begin{lemma}
Suppose that A1 and A3 hold and let the pair $(x^{k},\sigma_{k})$ be generated by Algorithm 4 such that $\sigma_{k}=0$ and $\|\nabla_{\Omega}^{I}f(x^{k})\|\geq\theta\|\nabla_{\Omega}f(x^{k})\|$. In addition, suppose that $\texttt{D}_{\texttt{type}}^{k}=\texttt{`SOL'}$ and $\texttt{Flag}=0$. Given $\alpha\in (0,1]$, if 
\begin{equation}
f_{x^{k}}(\alpha d^{k})>f_{x^{k}}(0)+\rho\alpha\langle\nabla f_{x^{k}}(0),d^{k}\rangle,
\label{eq:final1}
\end{equation}
then
\small
\begin{equation}
\alpha>\max\left\{\sqrt{3(1-2\rho)\eta/(L_{H}\|d^{k}\|}),2(1-\rho)\eta/L_{g}\right\}.
\label{eq:final2}
\end{equation}
\normalsize
\label{lem:2.2final}
\end{lemma}

\begin{proof}
Since $\sigma_{k}=0$ and $\|\nabla_{\Omega}^{I}f(x^{k})\|\geq\theta\|\nabla_{\Omega}f(x^{k})\|$, it follows from Step 2 of Algorithm 4 that Algorithm 5 is called. Then, as $\texttt{D}_{\texttt{type}}^{k}=\texttt{`SOL'}$ and $\texttt{Flag}=0$, it follows from Step 2.1 of Algorithm~5 and definition (\ref{eq:direction}) that $d^{k}=d_{1}^{k}=s^{k}$. Consequently, by (\ref{eq:minres1}), $\texttt{Flag}=0$ and (\ref{eq:flag}), we have
\begin{eqnarray}
\langle\nabla f_{x^{k}}(0),d^{k}\rangle &\leq & -\langle\nabla^{2}f_{x^{k}}(0)d^{k},d^{k}\rangle,
\label{eq:final3}\\
\langle\nabla^{2}f_{x^{k}}(0)d^{k},d^{k}\rangle&\geq& \eta_{k} \|d^{k}\|^{2}\geq \eta\|d^{k}\|^{2}.
\label{eq:final4}
\end{eqnarray}
It follows from A3, (\ref{eq:fx}), (\ref{eq:gx}), (\ref{eq:final3}), and (\ref{eq:final4}) that
\small
\begin{eqnarray}
f_{x^{k}}(\alpha d^{k})&\leq & 
f_{x^{k}}(0)+\alpha\langle\nabla f_{x^{k}}(0),d^{k}\rangle+\frac{\alpha}{2}\langle\nabla^{2}f_{x^{k}}(0)d^{k},d^{k}\rangle+\frac{L_{H}}{6}\alpha^{3}\|d^{k}\|^{3}.
\label{eq:final5}
\end{eqnarray}
\normalsize
Now, combining (\ref{eq:final3}), (\ref{eq:final1}) and (\ref{eq:final5}), we get
\small
\begin{eqnarray*}
(1-\rho)\alpha\langle\nabla^{2}f_{x^{k}}(0)d^{k},d^{k}\rangle&\leq& (1-\rho)\alpha\left(-\langle\nabla f_{x^{k}}(0),d^{k}\rangle\right)= \rho\alpha\langle\nabla f_{x^{k}}(0),d^{k}\rangle-\alpha\langle\nabla f_{x^{k}}(0),d^{k}\rangle\\
&<& f_{x^{k}}(\alpha d^{k})-f_{x^{k}}(0)-\alpha\langle\nabla f_{x^{k}}(0),d^{k}\rangle\\
&\leq &\frac{\alpha}{2}\langle\nabla^{2}f_{x^{k}}(0)d^{k},d^{k}\rangle+\frac{L_{H}}{6}\alpha^{3}\|d^{k}\|^{3}.
\end{eqnarray*}
\normalsize
Thus,
$\left(1/2-\rho\right)\alpha\langle\nabla^{2}f_{x^{k}}(0)d^{k},d^{k}\rangle<(L_{H}/6)\alpha^{3}\|d^{k}\|^{3}$,
and so, by (\ref{eq:final4}), 
\begin{equation}
\alpha>\sqrt{\frac{3(1-2\rho)}{L_{H}\|d^{k}\|}\frac{\langle\nabla^{2}f_{x^{k}}(0)d^{k},d^{k}\rangle}{\|d^{k}\|^{2}}}\geq\sqrt{\frac{3(1-2\rho)\eta}{L_{H}\|d^{k}\|}}.
\label{eq:final6}
\end{equation}
On the other hand, by A1 and (\ref{eq:gx}), we have
\begin{eqnarray}
f_{x^{k}}(\alpha d^{k})&\leq & 
f_{x^{k}}(0)+\alpha\langle\nabla f_{x^{k}}(0),d^{k}\rangle+\frac{L_{g}}{2}\alpha^{2}\|d^{k}\|^{2}.
\label{eq:final7}
\end{eqnarray}
Combining (\ref{eq:final7}) and (\ref{eq:final1}), we obtain
\[
\rho\alpha\langle\nabla f_{x^{k}}(0),d^{k}\rangle<\alpha\langle\nabla f_{x^{k}}(0),d^{k}\rangle+\frac{L_{g}}{2}\alpha^{2}\|d^{k}\|^{2},
\]
which together with (\ref{eq:final3}) and (\ref{eq:final4}) gives
\begin{equation}
\alpha>\frac{2(1-\rho)}{L_{g}}\left(-\frac{\langle\nabla f_{x^{k}}(0),d^{k}\rangle}{\|d^{k}\|^{2}}\right)\geq\frac{2(1-\rho)}{L_{g}}\frac{\langle\nabla^{2}f_{x^{k}}(0)d^{k},d^{k}\rangle}{\|d^{k}\|^{2}}\geq\frac{2(1-\rho)\eta}{L_{g}}.
\label{eq:final8}
\end{equation}
Finally, from (\ref{eq:final6}) and (\ref{eq:final8}) we see that (\ref{eq:final2}) is true.
\end{proof}

\begin{lemma}
\label{lem:2.3extra}
Suppose that A1 and A3 hold and let the pair $(x^{k},\sigma_{k})$ be generated by Algorithm 4 such that $\sigma_{k}=0$ and $\|\nabla_{\Omega}^{I}f(x^{k})\|\geq\theta\|\nabla_{\Omega}f(x^{k})\|$. In addition, suppose that $\texttt{D}_{\texttt{type}}^{k}=\texttt{`NPC'}$ and $\texttt{Flag}=0$. Given $\alpha\in (0,1]$, if  
\begin{equation}
f_{x^{k}}(\alpha d^{k})>f_{x^{k}}(0)+\rho\alpha\langle\nabla f_{x^{k}}(0),d^{k}\rangle
\label{eq:2.21extra}
\end{equation}
then
\small
\begin{equation}
\alpha>\max\left\{\sqrt{6(1-\rho)/(L_{H}\|d^{k}\|}),2(1-\rho)/L_{g}\right\}.
\label{eq:2.22extra}
\end{equation}
\normalsize
\end{lemma}

\begin{proof}
Since $\sigma_{k}=0$ and $\|\nabla_{\Omega}^{I}f(x^{k})\|\geq\theta\|\nabla_{\Omega}f(x^{k})\|$, it follows from Step 2 of Algorithm 4 that Algorithm 5 is called. Then, as $\texttt{D}_{\texttt{type}}^{k}=\texttt{`NPC'}$ and $\texttt{Flag}=0$, it follows from Step 2.1 of Algorithm~5 and definition (\ref{eq:direction}) that $d^{k}=d_{1}^{k}=r^{k}$. Consequently, by (\ref{eq:minres4}) and (\ref{eq:minres5}) we have
\begin{eqnarray} 
&\langle\nabla f_{x^{k}}(0),d^{k}\rangle = -\|d^{k}\|^{2},& \label{eq:final9} \\
&\langle\nabla^{2}f_{x^{k}}(0)d^{k},d^{k}\rangle < 0.& \label{eq:final10}
\end{eqnarray}
It follows from A3, (\ref{eq:fx}), (\ref{eq:gx}) and (\ref{eq:final10}) that
\small
\begin{eqnarray}
    f_{x^{k}}(\alpha d^{k})&\leq& 
    f_{x^{k}}(0)+\alpha\langle\nabla f_{x^{k}}(0),d^{k}\rangle+(L_{H}/6)\alpha^{3}\|d^{k}\|^{3}.
    \label{eq:2.24extra}
\end{eqnarray}
\normalsize
Combining (\ref{eq:2.21extra}) and (\ref{eq:2.24extra}), it follows that
\[
\rho\alpha\langle\nabla f_{x^{k}}(0),d^{k}\rangle<\alpha\langle\nabla f_{x^{k}}(0),d^{k}\rangle+(L_{H}/6)\alpha^{3}\|d^{k}\|^{3}.
\]
Thus, by (\ref{eq:final9}),
$\alpha^{2}>6(1-\rho)\left(-\langle\nabla f_{x^{k}}(0),d^{k}\rangle\right) / (L_{H}\|d^{k}\|^{3}) = 6(1-\rho)/(L_{H}\|d^{k}\|)$,
which implies that 
\small
\begin{equation}
\alpha>\sqrt{6(1-\rho) / (L_{H}\|d^{k}\|)}.
\label{eq:2.25extra}
\end{equation}
\normalsize
On the other hand, by A1, (\ref{eq:fx}) and (\ref{eq:gx}) we also have
\begin{eqnarray}
f_{x^{k}}(\alpha d^{k})&\leq& 
f(x^{k})+\alpha\langle\nabla f_{x^{k}}(0),d^{k}\rangle+(L_{g}/2)\alpha^{2}\|d^{k}\|^{2}.
\label{eq:2.26extra}
\end{eqnarray}
Then, combining (\ref{eq:2.21extra}) and (\ref{eq:2.26extra}), it follows that
\[
\rho\alpha\langle\nabla f_{x^{k}}(0),d^{k}\rangle<\alpha\langle\nabla f_{x^{k}}(0),d^{k}\rangle+(L_{g}/2)\alpha^{2}\|d^{k}\|^{2}.
\]
Therefore, by (\ref{eq:final9}), we get 
\begin{equation}
\alpha>2(1-\rho)(-\langle\nabla f_{x^{k}}(0),d^{k}\rangle) / (L_{g}\|d^{k}\|^{2}) = 2(1-\rho) / L_{g}.
\label{eq:2.27extra}
\end{equation}
In view of (\ref{eq:2.25extra}) and (\ref{eq:2.27extra}), we conclude that (\ref{eq:2.22extra}) is true.
\end{proof}

In view of Lemmas 2.7 and 2.8, we can now derive lower bounds for the functional decrease $f(x^{k})-f(x^{k+1})$ that occurs when $x^{k+1}$ is computed in Step 4 of Algorithm 5.

\begin{lemma}
Suppose that A1 and A3 hold. Then, whenever $x^{k+1}$ is computed at Step 4 of Algorithm~5, we have
\small
\setlength{\arraycolsep}{3pt}
\begin{equation}
f(x^{k})-f(x^{k+1})\geq\left\{\begin{array}{ll}
\frac{\rho(1-\rho)}{L_{g}}\left(\frac{a_{2}}{a_{1}}\right)^{2}\|\nabla_{\Omega}^{I}f(x^{k})\|^{2},&\text{if $\texttt{Flag}=1$,}\\
& \\
\frac{\rho}{2^{5/2}}\left(\frac{\eta}{L_{g}}\right)^{3/2}\sqrt{\frac{3(1-2\rho)}{L_{H}}}\|\nabla_{\Omega}^{I}f(x^{k})\|^{3/2},&\text{if $\texttt{Flag}=0$ and $\texttt{D}_{\texttt{type}}^{k}=\texttt{`SOL'}$,}\\
&\\
\frac{\rho}{2\left[\left(\frac{L_{g}}{\eta}\right)+1\right]^{3/2}}\sqrt{\frac{6(1-\rho)}{L_{H}}}\|\nabla_{\Omega}^{I}f(x^{k})\|^{3/2},&\text{if $\texttt{Flag}=0$ and $\texttt{D}_{\texttt{type}}^{k}=\texttt{`NPC'}$.}
\end{array}
\right.
\label{eq:edg1}
\end{equation}
\normalsize
Moreover, the number of evaluations of $f(\,\cdot\,)$ required to guarantee the fulfilment of (\ref{eq:2.21}) is bounded from above by 
$$|\log_{2}\left(\min\left\{(1-\rho)\eta / 2L_{g},(1-\rho)a_{2} / 2L_{g}a_{1}^{2}\right\}\right)|.$$
\label{lem:edg1}
\end{lemma}

\begin{proof}
Since $x^{k+1}$ is computed by Step 4 of Algorithm 5, it follows that $\ell_{k}>0$, and so
\begin{equation}
f_{x^{k}}((0.5)^{\ell_{k}-1}\alpha_{0}d^{k})>f_{x^{k}}(0)+\rho(0.5)^{\ell_{k}-1}\alpha_{0}\langle\nabla f_{x^{k}}(0),d^{k}\rangle.
\label{eq:edg2}
\end{equation}
Let us analyse separately the possible cases.
\\[0.2cm]
\noindent\textbf{Case 1:} $\texttt{Flag}=1$.
\\[0.2cm]
\noindent By Step 1 of Algorithm 5, $d^{k}$ is defined by the same correction procedure employed in Algorithm 2. Consequently, by Lemma 2.1 we have
\begin{equation}
\|d^{k}\|\leq a_{1}\|\nabla f_{x^{k}}(0)\|\quad\text{and}\quad \langle\nabla f_{x^{k}}(0),d^{k}\rangle\leq -a_{2}\|\nabla f_{x^{k}}(0)\|^{2}.
\label{eq:edg3}
\end{equation}
Then, by A1, Lemma 2.2 also applies to $d^{k}$. In particular, it follows from (\ref{eq:edg2}) that
\begin{equation}
(0.5)^{\ell_{k}-1}\alpha_{0}> 2(1-\rho)a_{2} / (L_{g}a_{1}^{2}).
\label{eq:edg4}
\end{equation}
Now, combining (\ref{eq:fx}), (\ref{eq:2.21}), the second inequality in (\ref{eq:edg3}), and (\ref{eq:edg4}), we obtain
\small
\begin{eqnarray*}
f(x^{k})-f(x^{k+1})&=& f_{x^{k}}(0)-f_{x^{k}}((0.5)^{\ell_{k}}\alpha_{0}d^{k})\geq\rho (0.5)^{\ell_{k}}\alpha_{0}\left(-\langle\nabla f_{x^{k}}(0),d^{k}\rangle\right)\\
&\geq &\frac{\rho}{2}(0.5)^{\ell_{k}-1}\alpha_{0}a_{2}\|\nabla f_{x^{k}}(0)\|^{2}\geq\frac{\rho(1-\rho)}{L_{g}}\left(\frac{a_{2}}{a_{1}}\right)^{2}\|\nabla f_{x^{k}}(0)\|^{2}\\
&\geq &\frac{\rho (1-\rho)}{L_{g}}\left(\frac{a_{2}}{a_{1}}\right)^{2}\|\nabla_{\Omega}^{I}f(x^{k})\|^{2},
\end{eqnarray*}
\normalsize
that is, (\ref{eq:edg1}) holds in this case.
\\[0.2cm]
\noindent\textbf{Case 2:} $\texttt{Flag}=0$ and $\texttt{D}_{\texttt{type}}^{k}=\texttt{`SOL'}$.
\\[0.2cm]
\noindent In this case, it follows from (\ref{eq:edg2}) and Lemma 2.7 that
\small
\begin{equation}
(0.5)^{\ell_{k}-1}\alpha_{0}>\max\left\{\sqrt{3(1-2\rho)\eta/(L_{H}\|d^{k})\|},2(1-\rho)\eta/L_{g}\right\}.
\label{eq:edg5}
\end{equation}
\normalsize
In addition, by (\ref{eq:flag}), (\ref{eq:minres1})-(\ref{eq:minres3}), and $\eta_{k}\geq\eta$, we also have 
\begin{eqnarray}
\langle\nabla f_{x^{k}}(0),d^{k}\rangle&\leq &-\langle\nabla^{2}f_{x^{k}}(0)d^{k},d^{k}\rangle,
\label{eq:edg6}\\
\langle\nabla^{2}f_{x^{k}}(0)d^{k},d^{k}\rangle&\geq&\eta\|d^{k}\|^{2},
\label{eq:edg7}\\
\langle\nabla^{2}f_{x^{k}}(0)r^{k},r^{k}\rangle&\geq&\eta\|r^{k}\|^{2},
\label{eq:edg8}\\
\|\nabla^{2}f_{x^{k}}(0)r^{k}\|&\leq&\eta_{k}\|\nabla^{2}f_{x^{k}}(0)d^{k}\|.
\label{eq:edg9}
\end{eqnarray}
Combining (\ref{eq:fx}), (\ref{eq:2.21}), (\ref{eq:edg6}) and (\ref{eq:edg7}), it follows that
\begin{eqnarray}
    f(x^{k})-f(x^{k+1})&=&f_{x^{k}}(0)-f_{x^{k}}((0.5)^{\ell_{k}}\alpha_{0}d^{k})\geq \rho (0.5)^{\ell_{k}}\alpha_{0}\left(-\langle\nabla f_{x^{k}}(0),d^{k}\rangle\right)\nonumber\\
    &\geq &\frac{\rho}{2}(0.5)^{\ell_{k}-1}\alpha_{0}\langle\nabla^{2}f_{x^{k}}(0)d^{k},d^{k}\rangle\geq \frac{\eta\rho}{2}\sqrt{\frac{3(1-2\rho)\eta}{L_{H}\|d^{k}\|}}\|d^{k}\|^{2}.
\label{eq:edg10}
\end{eqnarray}
On the other hand, by (\ref{eq:edg8}) we have
\[
\|r^{k}\|^{2}=\frac{\|r^{k}\|^{2}}{\langle\nabla^{2}f_{x^{k}}(0),r^{k}\rangle}\langle\nabla^{2}f_{x^{k}}(0)r^{k},r^{k}\rangle\leq\frac{\|\nabla^{2}f_{x^{k}}(0)r^{k}\|\|r^{k}\|}{\eta_{k}}.
\]
Then, dividing both sides by $\|r^{k}\|$ and using (\ref{eq:edg9}) and A1 we get
\[
\|r^{k}\|\leq\frac{\|\nabla^{2}f_{x^{k}}(0)r^{k}\|}{\eta_{k}}\leq\|\nabla^{2}f_{x^{k}}(0)d^{k}\|\leq\|\nabla^{2}f_{x^{k}}(0)\|\|d^{k}\|\leq L_{g}\|d^{k}\|.
\]
Thus,
\begin{eqnarray}
\|\nabla f_{x^{k}}(0)\|&\leq&
\|r^{k}\|+\|\nabla^{2}f_{x^{k}}(0)d^{k}\|\leq 2L_{g}\|d^{k}\|.
\label{eq:edg11}
\end{eqnarray}
Combining (\ref{eq:edg10}) and (\ref{eq:edg11}), it follows that
\small
\begin{eqnarray*}
    f(x^{k})-f(x^{k+1})&\geq &\frac{\eta\rho}{2}\sqrt{\frac{3(1-2\rho)\eta}{L_{H}\|d^{k}\|}}\frac{\|\nabla f_{x^{k}}(0)\|^{3/2}}{(2L_{g})^{3/2}}\\
    &\geq & \frac{\rho}{2^{5/2}}\left(\frac{\eta}{L_{g}}\right)^{3/2}\sqrt{\frac{3(1-2\rho)}{L_{H}}}\|\nabla_{\Omega}^{I}f(x^{k})\|^{3/2},
\end{eqnarray*}
\normalsize
that is, (\ref{eq:edg1}) holds in this case.
\\[0.2cm]
\noindent\textbf{Case 3:} $\texttt{Flag}=0$ and $\texttt{D}_{\texttt{type}}^{k}=\texttt{`NPC'}$.
\\[0.2cm]
\noindent In this case, it follows from (\ref{eq:edg2}) and Lemma 2.8 that
\begin{equation}
(0.5)^{\ell_{k}-1}\alpha_{0}>\max\left\{\sqrt{6(1-\rho)/(L_{H}\|d^{k}\|)},2(1-\rho)/L_{g}\right\}.
\label{eq:edg12}
\end{equation}
In addition, by (\ref{eq:flag}) and (\ref{eq:minres4}) we have
\begin{equation}
    \langle\nabla f_{x^{k}}(0),d^{k}\rangle=-\|d^{k}\|^{2},
    \label{eq:edg13}
\end{equation}
and
\begin{equation}
\|\nabla^{2}f_{x^{k}}(0)d^{k}\|>\eta_{k}\|\nabla^{2}f_{x^{k}}(0)s^{k}\|\geq \eta\|\nabla^{2}f_{x^{k}}(0)s^{k}\| .
    \label{eq:edg14}
\end{equation}
Combining (\ref{eq:fx}), (\ref{eq:2.3}), (\ref{eq:edg13}) and (\ref{eq:edg12}), it follows that
\small
\begin{eqnarray}
f(x^{k})-f(x^{k+1})&\geq& 
\rho (0.5)^{\ell_{k}}\alpha\left(-\langle\nabla f_{x^{k}}(0),d^{k}\rangle\right)=\frac{\rho}{2}(0.5)^{\ell_{k}-1}\alpha_{0}\|d^{k}\|^{2}\nonumber\\
&\geq &(\rho/2)\sqrt{6(1-\rho)/L_{H}}\|d^{k}\|^{3/2}.
\label{eq:edg15}
\end{eqnarray}
\normalsize
On the other hand, by (\ref{eq:edg14}) and A1, we have
\small
\begin{equation} \label{eq:edg16}
\begin{array}{lclcl}
\|\nabla f_{x^{k}}(0)\| & \leq & \|\nabla f_{x^{k}}(0)+r^{k}\|+\|r^{k}\| & = & \|\nabla^{2} f_{x^{k}}(0)s^{k}\|+\|d^{k}\| \\[2mm]
& < & \|\nabla^{2} f_{x^{k}}(0)d^{k}\| / \eta+\|d^{k}\| & \leq & \|\nabla^{2}f_{x^{k}}(0)\|\|d^{k}\| / \eta +\|d^{k}\| \\[2mm]
& \leq & (L_{g}/\eta+1) \|d^{k}\|. & & 
\end{array}
\end{equation}
\normalsize
Combining (\ref{eq:edg15}) and (\ref{eq:edg16}), it follows that
\small
\[
f(x^{k})-f(x^{k+1}) 
\geq 
\frac{\rho}{2}\sqrt{\frac{6(1-\rho)}{L_{H}}}\frac{\|\nabla f_{x^{k}}(0)\|^{3/2}}{(L_{g}/\eta+1)^{3/2}}
\geq \frac{\rho}{2(L_{g}/\eta+1)^{3/2}}\sqrt{\frac{6(1-\rho)}{L_{H}}}\|\nabla_{\Omega}^{I}f(x^{k})\|^{3/2},
\]
\normalsize
that is, (\ref{eq:edg1}) also holds in this case.
\\[0.2cm]
To conclude, notice that the number of evaluations of $f(\,\cdot\,)$ performed at Step 3 of Algorithm 5 is equal to $\ell_{k}+1$. Since $\alpha_{0}\in (0,1]$, it follows from (\ref{eq:edg3}), (\ref{eq:edg5}) and (\ref{eq:edg12}) that
\[
(0.5)^{\ell_{k}+1}\geq (0.5)^{\ell_{k}+1}\alpha_{0}>\min\left\{(1-\rho)\eta/(2L_{g}), (1-\rho)a_{2}/(2L_{g}a_{1}^{2}),(1-\rho)/(2L_{g})\right\}.
\]
Then, taking the logarithm on both sides and using the fact that $\eta\in (0,1]$, we obtain
\[
\ell_{k}+1\leq\left|\log_{2}\left(\min\left\{(1-\rho)\eta/(2L_{g}),(1-\rho)a_{2}/(2L_{g}a_{1}^{2})\right\}\right)\right|.
\]
\end{proof}

The lemma below provides an upper bound on the maximum stepsize that satisfies the Armijo condition when the direction $d^{k}$ is certified with $\texttt{Flag}=0$. This result will be used in the sequel to derive an upper bound on the number of function evaluations required in the extrapolation procedure in Step 6 of Algorithm 5.

\begin{lemma}
\label{lem:2.4extra}
Suppose that A1-A3 hold and let $x^{k}$ be an iterate generated by Algorithm 4 such that $\|\nabla_{\Omega}^{I}f(x^{k})\|\geq\theta\|\nabla_{\Omega} f(x^{k})\|$. Suppose that $d^{k}$ is computed by Algorithm 5 and that $\texttt{Flag}=0$. Given $\alpha>0$, if $\|\nabla_{\Omega} f(x^{k})\|\geq\epsilon$ and
\begin{equation}
f_{x^{k}}(\alpha d^{k})\leq f_{x^{k}}(0)+\rho\alpha\langle\nabla f_{x^{k}}(0),d^{k}\rangle,
\label{eq:2.28extra}
\end{equation}
then
\begin{equation}
\alpha\leq \max\left\{2 L_{g}/\eta,(L_{g}/\eta+1)^{2}\right\}(f(x^{0})-f_{\mathrm{low}})\rho^{-1}\theta^{-2}\epsilon^{-2}.
\label{eq:2.29extra}    
\end{equation}
\end{lemma}

\begin{proof}
If $\texttt{D}_{\texttt{type}}^{k}=\texttt{`SOL'}$, then it follows from (\ref{eq:2.28extra}), (\ref{eq:edg6}), (\ref{eq:edg7}) and (\ref{eq:edg11}) that
\[
\begin{array}{lclcl}
f(x^{0})-f_{\mathrm{low}} & \geq & f_{x^{k}}(0)-f_{x^{k}}(\alpha d^{k}) & \geq & \rho\alpha\left(-\langle\nabla f_{x^{k}}(0),d^{k}\rangle\right) \\[2mm]
& \geq & \rho\alpha\langle\nabla^{2}f_{x^{k}}(0)d^{k},d^{k}\rangle & \geq &  \rho\alpha\eta\|d^{k}\|^{2}\\[2mm]
& \geq & (\rho/2)(\eta/L_g)\alpha\|\nabla_{\Omega}^{I}f(x^{k})\|^{2} & \geq & (\theta^2 \rho / 2)(\eta/L_g)\epsilon^{2}\alpha.
\end{array}
\]

Thus, $\alpha\leq 2(L_g/\eta)\left(f(x^{0})-f_{\mathrm{low}}\right)\rho^{-1}\theta^{-2}\epsilon^{-2}$, which means that (\ref{eq:2.29extra}) holds in this case. 

Now, suppose that $\texttt{D}_{\texttt{type}}^{k}=\texttt{`NPC'}$. Then, by (\ref{eq:2.28extra}), (\ref{eq:edg13}) and (\ref{eq:edg16}) we have
\[
f(x^{0})-f_{\mathrm{low}} \geq \rho\alpha\|d^{k}\|^{2}\geq \frac{\rho\alpha}{(L_g/\eta+1)^2}\|\nabla f_{x^{k}}(0)\|^{2}=\frac{\rho\alpha}{(L_g/\eta+1)^2}\|\nabla_{\Omega}^{I}f(x^{k})\|^{2}
\]
\[
\geq \frac{\rho\alpha}{(L_g/\eta+1)^2}\theta^{2}\|\nabla_{\Omega}f(x^{k})\|^{2}\geq\frac{\rho\theta^{2}}{(L_g/\eta+1)^2}\epsilon^{2}\alpha.
\]
Therefore, $\alpha\leq \left(L_{g}/\eta+1\right)^{2}\left(f(x^{0})-f_{\mathrm{low}}\right)\rho^{-1}\theta^{-2}\epsilon^{-2}$, which means that (\ref{eq:2.29extra}) also holds in this case. 
\end{proof}

The next lemma gives an upper bound of $\mathcal{O}\left(|\log_{2}(\epsilon)|\right)$ for the number of evaluations of $f(\,\cdot\,)$ required by the extrapolation procedure in Step 6 of Algorithm 5.

\begin{lemma}
\label{lem:2.6extra}
Suppose that assumptions A1-A3 hold, and that $x^{k+1}$ is computed in Step 6 of Algorithm 5. Then, the number of evaluations of $f(\,\cdot\,)$ required to guarantee the fulfilment of (\ref{eq:extrapolation}) is bounded from above by
\begin{equation}
1+\log_{2}\left(\max\left\{L_{g}/\eta,\left(L_{g}/\eta+1\right)^{2}\right\}(f(x^{0})-f_{\mathrm{low}})\rho^{-1}\theta^{-2}\epsilon^{-2}\right).
\label{eq:bound}
\end{equation}
\end{lemma}

\begin{proof}
The execution of Step 6 of Algorithm 5 implies that $f_{x^{k}}(\alpha_{0}d^{k})\leq f_{x^{k}}(0)$. Therefore, $\alpha_{0}=1$. Indeed, if we had $\alpha_{0}=t_{\max}<1$, then we would have had $f_{x^{k}}(t_{\max}d^{k})\leq f_{x^{k}}(0)$, and the method would have stopped at Step 2.3, which was not the case. Notice that the number of evaluations of $f(\,\cdot\,)$ required to fulfill (\ref{eq:extrapolation}) is equal to $j_{k}+1$. Since $x^{k+1}$ is computed by Algorithm 5, we must have had $\|\nabla_{\Omega}^{I}f(x^{k})\|>\theta\|\nabla_{\Omega} f(x^{k})\|$, and $\|\nabla_{\Omega} f(x^{k})\|>\epsilon$. Moreover, by design, we also have $f(x^{k})\leq f(x^{0})$. Then, the definition of $j_{k}$, $\alpha_{0}=1$, and Lemma \ref{lem:2.4extra} imply that
\[
2^{j_{k}}=2^{j_{k}}\alpha_{0}\leq \max\left\{L_{g}/\eta,\left(L_{g}/\eta+1\right)^{2}\right\}(f(x^{0})-f_{\mathrm{low}})\rho^{-1}\theta^{-2}\epsilon^{-2}.
\]
and so, taking the logarithm, we conclude that $j_{k}+1$ is bounded from above by the number in~(\ref{eq:bound}).
\end{proof}

The following two lemmas establish lower bounds of $\mathcal{O}\left(\|\nabla_{\Omega}^{I}f(x^{k})\|^{3/2}\right)$ for the functional decrease $f(x^{k})-f(x^{k+1})$ obtained when $x^{k+1}$ is computed in Steps 6.1 and 6.2 of Algorithm 5.

\begin{lemma}
Suppose that A1 and A3 hold. Then, whenever $x^{k+1}$ is computed in Step 6.1 of Algorithm 5, we have
\[
f(x^{k})-f(x^{k+1})\geq\left\{\begin{array}{ll}
\frac{\rho}{2^{5/2}}\left(\frac{\eta}{L_{g}}\right)^{3/2}\sqrt{\frac{3(1-2\rho)}{L_{H}}}\|\nabla_{\Omega}^{I}f(x^{k})\|^{3/2},&\text{if $\texttt{D}_{\texttt{type}}^{k}=\texttt{`SOL'}$,}\\
\frac{\rho}{2\left[\left(\frac{L_{g}}{\eta}\right)+1\right]^{3/2}}\sqrt{\frac{6(1-\rho)}{L_{H}}}\|\nabla_{\Omega}^{I}f(x^{k})\|^{3/2},&\text{if $\texttt{D}_{\texttt{type}}^{k}=\texttt{`NPC'}$.}
\end{array}
\right.
\]
\end{lemma}

\begin{proof}
As $x^{k+1}$ is computed in Step 6.1, this means that $\texttt{Flag}=0$; otherwise, the Algorithm 5 would have stopped at Step 5. Since $x^{k}+2^{j_{k}+1}\alpha_{0}Q(x^{k})d^{k}\in\Omega$, it follows from the definition of $j_{k}$, that 
\begin{equation}
f_{x^{k}}(2^{j_{k}+1}\alpha_{0}d^{k})>f_{x^{k}}(0)+\rho 2^{j_{k}+1}\alpha_{0}\langle\nabla f_{x^{k}}(0),d^{k}\rangle.
\label{eq:edgfinal2}
\end{equation}
If $\texttt{D}_{\texttt{type}}^{k}=\texttt{`SOL'}$, then it follows from Lemma 2.7 that $2^{j_{k}+1}\alpha_{0}>\sqrt{\frac{3(1-2\rho)\eta}{L_{H}\|d^{k}\|}}$. Thus, by the same reasoning used in the proof of Lemma 2.9 (Case 2), we see that
\[
f(x^{k})-f(x^{k+1})\geq\frac{\rho}{2^{5/2}}\left(\frac{\eta}{L_{g}}\right)^{3/2}\sqrt{\frac{3(1-2\rho)}{L_{H}}}\|\nabla_{\Omega}^{I}f(x^{k})\|^{3/2}.
\]
On the other hand, if $\texttt{D}_{\texttt{type}}^{k}=\texttt{`NPC'}$, then it follows from (\ref{eq:edgfinal2}) and Lemma 2.8 that $2^{j_{k}+1}\alpha_{0}>\sqrt{6(1-\rho)/(L_{H}\|d^{k}\|)}$. Thus, as in the proof of Lemma 2.9 (Case 3), it follows that 
\[
f(x^{k})-f(x^{k+1}) \geq \frac{\rho}{2(L_{g}/\eta+1)^{3/2}}\sqrt{\frac{6(1-\rho)}{L_{H}}}\|\nabla_{\Omega}^{I}f(x^{k})\|^{3/2},
\]
which concludes the proof.
\end{proof}

\begin{lemma}
Suppose that A1 and A3 hold. Then, whenever $x^{k+1}$ is computed in Step 6.2 of Algorithm 5, we have
\begin{equation}
f(x^{k}) - f(x^{k+1}) \geq 
\left\{
\begin{array}{ll}
\frac{\rho}{2^{5/2}}\left(\frac{\eta}{L_{g}}\right)^{3/2}\sqrt{\frac{3(1-2\rho)}{L_{H}}}\|\nabla_{\Omega}^{I}f(x^{k})\|^{3/2}, 
& \text{if } \texttt{D}_{\texttt{type}}^{k}=\texttt{`SOL'},\\
\frac{\rho}{2(L_g/\eta+1)^{3/2}}\sqrt{\frac{6(1-\rho)}{L_{H}}}\|\nabla_{\Omega}^{I}f(x^{k})\|^{3/2}, 
& \text{if } \texttt{D}_{\texttt{type}}^{k}=\texttt{`NPC'}.
\end{array}
\right.
\label{eq:edgfinal3}
\end{equation}
\end{lemma}

\begin{proof}
As $x^{k+1}$ is computed in Step 6.2, this means that $\texttt{Flag}=0$ and 
\begin{equation}
f_{x^{k}}(t_{\max}d^{k})>f_{x^{k}}(0)>f_{x^{k}}(0)+\rho t_{\max}\langle\nabla f_{x^{k}}(0),d^{k}\rangle.
\label{eq:edgfinal4}
\end{equation}
Since $x^{k}+2^{j_{k}}\alpha_{0}Q(x^{k})d^{k}\in\Omega$ and $x^{k}+2^{j_{k}+1}\alpha_{0}Q(x^{k})d^{k}\notin\Omega$\normalsize, we also have the inequality $2^{j_{k+1}}\alpha_{0}>t_{\max}$. Then, it follows from (\ref{eq:edgfinal4}) and Lemmas 2.7 and 2.8 that 
\small
\[
2^{j_{k}}\alpha_{0} > t_{\max} \geq
\left\{
\begin{array}{ll} 
\sqrt{3(1-2\rho)/(L_{H}\|d^{k}\|)}, 
& \text{if } \texttt{D}_{\texttt{type}}^{k}=\texttt{`SOL'},\\
\sqrt{6(1-\rho)/(L_{H}\|d^{k}\|)}, 
& \text{if } \texttt{D}_{\texttt{type}}^{k}=\texttt{`NPC'}.
\end{array}
\right.
\]
\normalsize
Thus, by following the same reasoning in the proof of Lemma 2.9 (Cases 2 and 3), we conclude that~(\ref{eq:edgfinal3}) is true.
\end{proof}

As shown in the lemma below, if $x^{k+1}$ is computed by Algorithm 6, the objective function decreases by at least a constant multiple of $\|\nabla_{\Omega}f(x^{k})\|^{3/2}$.

\begin{lemma} Suppose that A1 and A3 hold. Then, whenever $x^{k+1}$ is computed by Algorithm 6, we have
\small
\begin{equation}
f(x^{k})-f(x^{k+1}) \geq \alpha \left( \|\nabla_{\Omega}f(x^{k+1})\| / \left(L_{H}+6(L_{H}+\alpha)+\gamma\right) \right)^{\frac{3}{2}}.
    \label{eq:gg_extra2}
\end{equation}
\normalsize
Moreover, the number of evaluations of $f(\,\cdot\,)$ required to guarantee the fulfillment of (\ref{eq:gg_extra1}) is bounded from above by $1+\log_{2}\left(\max\left\{1,(L_{H}+\alpha)/M\right\}\right)$.
    \label{lem:2.4}
\end{lemma}

\begin{proof}
    Inequality (\ref{eq:gg_extra2}) follows from Lemma 3.4 in \cite{BM}. Regarding the upper bound on the number of function evaluations, note that Algorithm 6 requires $\ell_{k}+1$ evaluations of $f(\,\cdot\,)$. Thus, let us first show that
    \begin{equation}
     2^{\ell_{k}}M_{k}\leq\max\left\{M_{k},2(L_{H}+\alpha)\right\},
    \label{eq:gg_extra4}
    \end{equation}
    where, by design, $\ell_{k}$ is the smallest nonnegative integer for which (\ref{eq:gg_extra1}) is satisfied. If $\ell_{k}=0$, then $2^{\ell_{k}}M_{k}=M_{k}$ and then (\ref{eq:gg_extra4}) clearly holds. Suppose that $\ell_{k}>0$. By Lemma 3.2 in \cite{BM}, (\ref{eq:gg_extra1}) is satisfied whenever $2^{\ell}M_{k}\geq L_{H}+\alpha$. This means that we must have $2^{\ell_{k}-1}M_{k}<L_{H}+\alpha$, because, otherwise, (\ref{eq:gg_extra1}) would had been satisfied by some $\ell$ with $0\leq \ell\leq\ell_{k}-1$, contradicting the definition of $\ell_{k}$. Therefore, 
    \[
        2^{\ell_{k}}M_{k}=2\left(2^{\ell_{k}-1}M_{k}\right)<2(L_{H}+\alpha)\leq\max\left\{M_{k},2(L_{H}+\alpha)\right\},
    \]
    concluding the proof of (\ref{eq:gg_extra4}). Finally, from (\ref{eq:gg_extra4}) and $M_{k}\geq M$, it follows that $\ell_{k}+1\leq 1+\log_{2}\left(\max\left\{1,(L_{H}+\alpha)/M\right\}\right)$.
\end{proof}

Let $\left\{x^{k}\right\}_{k\geq 0}$ be generated by Algorithm 4, and denote by $T(\epsilon)$ the first hitting time to the set of $\epsilon$-approximate stationary points of $f(\,\cdot\,)$ with respect to $\Omega$, as in (\ref{eq:hitting}). Given $j\in\left\{0,\dots,T(\epsilon)-1\right\}$, consider the following sets:
\[
\begin{array}{rcl}
\mathcal{S}_{j}^{(1)} &=& \{k\in\{0,\dots,j\}:\sigma_{k+1}=0\},\\
\mathcal{S}_{j}^{(2)} &=& \{k\in\{0,\dots,j\}:\sigma_{k+1}=1\},\\
\mathcal{U}_{j} &=& \{k\in\{0,\dots,j\}:\sigma_{k+1}=2\}.
\end{array}
\]
The next theorem establishes that $T(\epsilon)\leq\mathcal{O}\left(n\epsilon^{-3/2}\right)$, i.e., Algorithm 4 needs no more than $\mathcal{O}(n\epsilon^{-3/2})$ iterations to find $\epsilon$-approximate stationary points.

\begin{theorem}
\label{thm:edg}
Suppose that A1-A3 hold, and let $\left\{x^{k}\right\}_{k=0}^{T(\epsilon)}$ be generated by Algorithm 4. Then
\begin{equation}
T(\epsilon)\leq (n+1)+2(n+1)\left(f(x^{0})-f_{\mathrm{low}}\right)\kappa_{2}^{-1}\epsilon^{-3/2},
\label{eq:edg17}
\end{equation}
where
\small
\begin{equation}
\kappa_{2}=\min\left\{\frac{\rho\theta^{\frac{3}{2}}}{2\left[\frac{L_{g}}{\eta}+1\right]^{\frac{3}{2}}}\sqrt{\frac{6(1-\rho)}{L_{H}}},\frac{\rho\theta^{\frac{3}{2}}}{2^{\frac{5}{2}}}\left(\frac{\eta}{L_{g}}\right)^{\frac{3}{2}}\sqrt{\frac{3(1-2\rho)}{L_{H}}},\frac{\alpha}{\left[L_{H}+6(L_{H}+\alpha)+\gamma\right]^{\frac{3}{2}}}\right\}.
\label{eq:edg18}
\end{equation}
\normalsize
\end{theorem}
\begin{proof}
If $T(\epsilon)\leq 1$, then (\ref{eq:edg17}) is true. Thus, suppose that $T(\epsilon)\geq 2$ and let $k\in\left\{0,\dots,T(\epsilon)-1\right\}$. By (\ref{eq:hitting}) we have
\begin{equation}
\|\nabla_{\Omega}f(x^{k})\|>\epsilon.
\label{eq:edg19}
\end{equation}
If $k\in\mathcal{S}_{T(\epsilon)-1}^{(1)}$, then $\sigma_{k+1}=0$. This means that $x^{k+1}$ was computed either by Algorithm 6, at Step~4 of Algorithm 5 with $\texttt{Flag}=0$, or at Steps 6.1 or 6.2 of Algorithm 5. In either case, Lemmas 2.9, 2.12 and 2.13 together with (\ref{eq:edg19}) imply that

\begin{equation}
f(x^{k})-f(x^{k+1})\geq\kappa_{2}\epsilon^{3/2},
\label{eq:edg20}
\end{equation}
where $\kappa_{2}$ is defined in (\ref{eq:edg18}). Moreover, we have
\begin{equation}
f(x^{\ell+1})\leq f(x^{\ell}),\quad\ell=0,\dots,T(\epsilon)-1.
\label{eq:edg21}
\end{equation}
Thus, combining A2, (\ref{eq:edg21}) and (\ref{eq:edg20}), it follows that
\small
\begin{eqnarray*}
    f(x^{0})-f_{\mathrm{low}}\geq f(x^{0})-f(x^{T(\epsilon)})
    \geq  \sum_{k\in\mathcal{S}_{T(\epsilon)-1}^{(1)}}f(x^{k})-f(x^{k+1})
    \geq  \left|\mathcal{S}_{T(\epsilon)-1}^{(1)}\right|\kappa_{2}\epsilon^{3/2},
\end{eqnarray*}
\normalsize
and so
\small
\begin{equation}
\left|\mathcal{S}_{T(\epsilon)-1}^{(1)}\right|\leq\left(f(x^{0})-f_{\mathrm{low}}\right)\kappa_{2}^{-1}\epsilon^{-3/2}.
\label{eq:edg22}
\end{equation}
\normalsize
On the other hand, if $k\in\mathcal{S}_{T(\epsilon)-2}^{(2)}$, then $\sigma_{k+1}=1$. Consequently, by Step 2 of Algorithm 4, $x^{k+2}$ is computed by Algorithm 6, and so $\sigma_{k+2}=0$. This means that every iteration in $\mathcal{S}_{T(\epsilon)-2}^{(2)}$ is followed by one iteration in $\mathcal{S}_{T(\epsilon)-1}^{(1)}$. Thus the cardinality of $\mathcal{S}_{T(\epsilon)-2}^{(2)}$ is not bigger than that of $\mathcal{S}_{T(\epsilon)-1}^{(1)}$. Consequently,
\small
\begin{equation}
\left|\mathcal{S}_{T(\epsilon)-1}^{(2)}\right|\leq 1+\left|\mathcal{S}_{T(\epsilon)-1}^{(1)}\right|.
\label{eq:edg23}
\end{equation}
\normalsize
Finally, if $k\in\mathcal{U}_{T(\epsilon)-1}$, then $x^{k+1}$ is computed at Step 2.2, Step 2.3, or Step 6.3 of Algorithm~5. This means that $x^{k+1}\in\partial\mathcal{F}(x^{k})$ and so $\left|\mathcal{I}(x^{k+1})\right|\leq\left|\mathcal{I}(x^{k})\right|-1$. Therefore, there can be at most $n$ consecutive iterations of Algorithm 4 with $k\in\mathcal{U}_{T(\epsilon)-1}$. In the worst case, each iteration in $\mathcal{S}_{T(\epsilon)-1}^{(1)}\cup\mathcal{S}_{T(\epsilon)-1}^{(2)}$ would be followed by $n$ consecutive iterations in $\mathcal{U}_{T(\epsilon)-1}$. Then, we have
\begin{equation}
\left|\mathcal{U}_{T(\epsilon)-1}\right|\leq n\left|\mathcal{S}_{T(\epsilon)-1}^{(1)}\cup\mathcal{S}_{T(\epsilon)-1}^{(2)}\right|=n\left(\left|\mathcal{S}_{T(\epsilon)-1}^{(1)}\right|+\left|\mathcal{S}_{T(\epsilon)-1}^{(2)}\right|\right).
\label{eq:edg24}
\end{equation}
Then, combining (\ref{eq:edg22}), (\ref{eq:edg23}) and (\ref{eq:edg24}), we conclude that
\small
\begin{eqnarray*}
    T(\epsilon)&=&\left|\mathcal{S}_{T(\epsilon)-1}^{(1)}\right|+\left|\mathcal{S}_{T(\epsilon)-1}^{(2)}\right|+\left|\mathcal{U}_{T(\epsilon)-1}\right|\leq \left|\mathcal{S}_{T(\epsilon)-1}^{(1)}\right|+1+\left|\mathcal{S}^{(1)}_{T(\epsilon)-1}\right|+n\left(\left|\mathcal{S}_{T(\epsilon)-1}^{(1)}\right|+\left|\mathcal{S}_{T(\epsilon)-1}^{(2)}\right|\right)\\
    &\leq & (n+1)+2(n+1)\left|\mathcal{S}_{T(\epsilon)-1}^{(1)}\right|\leq (n+1)+2(n+1)\left(\frac{f(x^{0})-f_{\mathrm{low}}}{\kappa_{2}}\right)\epsilon^{-3/2}.
\end{eqnarray*}
\end{proof}
\normalsize
\begin{remark}
In view of Lemmas 2.9 and 2.11, the computation of $x^{k+1}$ by Algorithm 5 requires no more than
\scriptsize
\[
\max\left\{\left|\log_{2}\left(\min\left\{\frac{(1-\rho)\eta}{2L_{g}},\frac{(1-\rho)a_{2}}{2L_{g}a_{1}^{2}}\right\}\right)\right|,\left|\log_{2}\left(\frac{\max\left\{\left(\frac{L_{g}}{\eta}\right),\left[\left(\frac{L_{g}}{\eta}\right)+1\right]^{2}\right\}(f(x^{0})-f_{\mathrm{low}})}{\rho\theta^{2}\epsilon^{2}}\right)\right|\right\}
\]
\normalsize
plus $m+5$ evaluations of $f(\,\cdot\,)$. On the other hand, by Lemma 2.14, the computation of $x^{k+1}$ by Algorithm 6 requires no more than $1+\log_{2}\left(\max\left\{1,(L_{H}+\alpha) / M\right\}\right)$ evaluations of $f(\,\cdot\,)$. Additionally, each execution of Algorithm 5 or Algorithm 6 requires the evaluation of one gradient and one Hessian of $f(\,\cdot\,)$. In summary, each iteration of Algorithm 4 requires at most $\mathcal{O}\left(|\log_{2}(\epsilon)|\right)$ calls to the oracle. Therefore, it follows from Theorem 2.15 that Algorithm 4 takes no more than $\mathcal{O}\left(n|\log_{2}(\epsilon)|\epsilon^{-3/2}\right)$ calls to the oracle to find an $x^{k}$ such that $\|\nabla_{\Omega}f(x^{k})\|\leq\epsilon$.
\end{remark}

\section{Numerical experiments} \label{sec3}

In this section, we present numerical results to evaluate the performance of the introduced algorithms. Hereafter, we will call Algorithm~P the algorithm comprising Algorithms~1, 2 and~3 and Algorithm~T the algorithm comprising Algorithms~4, 5 and~6. The letter~P indicates the practical appeal of the first one, with worst-case complexity $\mathcal{O}(n\epsilon^{-2})$, while the letter~T indicates the theoretical concern behind the development of the second algorithm, with worst-case complexity $\mathcal{O}(n|\log_2(\epsilon)|\epsilon^{-3/2})$. For both algorithms, we used as stop criterion $\| \nabla_{\Omega} f(x) \|_{\infty} \leq \epsilon = 10^{-8}$. Other stopping criteria related to maximum iterations and lack of progress exist, which are identical in the two algorithms. We also consider a CPU time limit of 10 minutes for each pair method/problem. In the following we describe some implementation details.

We implemented Algorithms~P and~T and MINRES in Fortran~90. Codes are available for download at \url{http://www.ime.usp.br/~egbirgin/}. The two methods were evaluated using all unconstrained and bound-constrained problems from the most updated version of the CUTEst collection~\cite{cutest} (version 2.4.0). There are 313 unconstrained problems and 162 bound-constrained problems in this release, for a total of 475 problems. We considered all the problems with their default dimension and the given starting point~$x^0$. The smallest problem has 1 variable, the largest problem has 192{,}627 variables, and the quartiles of the number of variables are $Q_1=4$, $Q_2=50$, and $Q_3=5{,}000$. All experiments were performed on a computer with a 5.2 GHz Intel Core i9-12900K and 5.1 GHz Intel Core i9-12900K processor and 9 128 GB of 32000 MHz DDR4 RAM, running Ubuntu 23.04. The codes were compiled by the GNU Fortran compiler GCC (version 12.3.0) with the -O3 optimization directive enabled.

\subsection{Implementation details} \label{sec31}

\subsubsection{When MINRES encounters a non-positive curvature direction} \label{sec311}

When MINRES is used in Step~1 of Algorithms~2 and~5, it returns an approximate solution $s^k$ plus $\texttt{D}_{\texttt{type}}^{k}$ = `NPC' or $\texttt{D}_{\texttt{type}}^{k}$ = `SOL'. The second case means that the Newtonian linear system has been solved with the desired tolerance (which will be detailed later). If this happens, $d_1^k$ in Algorithms~2 and~5 gets the computed solution~$s^k$. On the other hand, $\texttt{D}_{\texttt{type}}^{k}$ = `NPC' means that MINRES found a non-positive curvature direction while solving the linear system. If this happened when being called by Algorithm~5, $d_1^k$ gets the residue of the linear system, i.e. $d_1^k = -(\nabla^2 f_{x^k}(0) s^k + \nabla f_{x^k}(0))$. In the implementation of Algorithm~2, we evaluated two possibilities for the case where MINRES returns $\texttt{D}_{\texttt{type}}^{k}$ = `NPC'. The first is the one used in Algorithm~5, which is to consider $d_1^k = -(\nabla^2 f_{x^k}(0) s^k + \nabla f_{x^k}(0))$. The second is to consider $d_1^k=s^k$ if $s^k \neq 0$, i.e.\ the approximate solution itself, and $d_1^k=-\nabla f_{x^k}(0)$ if $s^k=0$. (Note that $s^k=0$ only when MINRES detects a non-positive curvature direction in its first iteration). Both options will be evaluated numerically below.

\subsubsection{Tolerance in solving Newtonian systems using MINRES} \label{sec312}

When MINRES is used in Step~1 of Algorithm~5, the required accuracy is dynamically determined by $\eta_k \geq \eta \equiv \epsilon$, updated according to~\eqref{eq:eta}, where its initial value $\eta_0$ and its update factor $\tau \in (0,1)$ are given parameters. In Algorithm~2, we use a dynamic tolerance $\epsilon^{\mathrm{\textsc{mr}}}_k$ borrowed from~\cite[p.113]{bmgencan}. For $k=1$ the tolerance is $\epsilon^{\mathrm{\textsc{mr}}}_1 = \epsilon^{\mathrm{\textsc{mr}}}_{\mathrm{ini}}$, where $\epsilon^{\mathrm{\textsc{mr}}}_{\mathrm{ini}} \geq \epsilon^{\mathrm{\textsc{mr}}}_{\mathrm{end}} \equiv \epsilon$ is a given parameter, and the idea is that in the last iteration the tolerance will be $\epsilon^{\mathrm{\textsc{mr}}}_{\mathrm{end}}$. To do this, at iteration $k \geq 1$ we use a tolerance whose value varies linearly with $\log_{10}(\|\nabla_{\Omega} f(x^k)\|)$, i.e., 
$\epsilon^{\mathrm{\textsc{mr}}}_k = \sqrt{10^{a \log_{10} ( | \nabla_\Omega f(x^k) |^2 ) + b}}$,
where
$a = \log_{10}(\epsilon^{\mathrm{\textsc{mr}}}_{\mathrm{end}} / \epsilon^{\mathrm{\textsc{mr}}}_{\mathrm{ini}}) / \log_{10}(\epsilon / | \nabla_\Omega f(x^0) | )$
and
$b = 2 \log_{10}(\epsilon^{\mathrm{\textsc{mr}}}_{\mathrm{ini}}) - a \log_{10} ( | \nabla_\Omega f(x^0) |^2 )$. 

\subsubsection{Optional extrapolations} \label{sec313}

Algorithm~P (Steps 2.2, 2.3, and 4 of Algorithm~2) and Algorithm~T (Steps 2.2, 2.3, 5, and 6.3 of Algorithm~5) attempt to improve the current point by extrapolation. These attempts are of limited effort by definition, so they do not affect the order of the complexity of the algorithms. For this reason, their practical influence on the performance of the methods must be determined numerically. The extrapolation consists of
\textbf{(i)} set $u_k=0$; 
\textbf{(ii)} while $u_k + 1\leq m$ and $f(P_{\Omega}(x^k + 2^{u_k+1} \alpha_0 Q(x^k) d^k)) \leq f(P_{\Omega}(x^k + 2^{u_k} \alpha_0 Q(x^k) d^k))$, set $u_k = u_k + 1$; \textbf{(iii)} define $x^{k+1} = P_{\Omega}(x^k +2^{u_k} \alpha_0 Q(x^k) d^k)$.

\subsubsection{Backtracking, Barzilai-Borwein stepsize, and other details} \label{sec314}

In Step 3 of Algorithm 2, Step 2 of Algorithm 3 and Step 3 of Algorithm 5, in practice, we use quadratic interpolation with safeguards to find a step $t_k>0$ such that the corresponding sufficient descent condition holds. The description of the algorithms refers to a step that is a power of 2 just to simplify the description. This change has no significance on the theoretical results.

In Algorithm 3, for $k \geq 1$ and whenever \small $(x^k - x^{k-1})^T (\nabla f_{x^k}(0) - \nabla f_{x^{k-1}}(0)) > 0$\normalsize, we compute
\small
\[
\lambda_k^{\mathrm{spg}} = \max \left\{ \lambda_{\min}^{\mathrm{spg}}, \min\left\{ \frac{(x^k - x^{k-1})^T (x^k - x^{k-1})}{(x^k - x^{k-1})^T (\nabla f_{x^k}(0) - \nabla f_{x^{k-1}}(0))}, \lambda_{\max}^{\mathrm{spg}} \right\} \right\}.
\]
\normalsize
In the other cases, $\lambda_k^{\mathrm{spg}} \in [\lambda_{\min}^{\mathrm{spg}},\lambda_{\max}^{\mathrm{spg}}]$ is arbitrary and we considered
\small
\[
\lambda_k^{\mathrm{spg}} = \max \left\{ \lambda_{\min}^{\mathrm{spg}}, \min\left\{ \frac{\max\{ 1, \| x^k \|_\infty \}}{\| \nabla_{\Omega}f(x^k) \|_\infty}, \lambda_{\max}^{\mathrm{spg}} \right\} \right\}.
\]
\normalsize

In Algorithm~6, the approximate solution to the subproblem in Step~2 is calculated using the projected gradient method. In addition, the representation of the regularization parameter with the term $2^\ell M$ in the iteration $\ell$, where $M>0$ is a parameter of Algorithm~6, is only a simplification for the presentation of the algorithm. In practice, the regularization parameter is represented by $\omega$. When $\ell=0$, we consider $\omega=0$. In Step~3, if sufficient decent is not obtained, together with the operation $\ell := \ell+1$, we update $\omega$ by making $\omega := \max\{ \omega_{\min}, \zeta \omega\}$, where $\omega_{\min}>0$ and $\zeta>1$ are parameters of the algorithm. In practice, we consider $\omega_{\min}=10^{-6}$ and $\zeta=10$, which are common values in regularized methods.

\subsection{Evaluation of Algorithm~P and Algorithm~T and their alternatives} \label{sec32}

This section compares different variants of Algorithms P and T, as well as the best variant of each algorithm. When comparing two algorithms, we first compare their robustness. In this study, we examine both unconstrained and bound-constrained problems. The algorithms we consider produce feasible iterates. Thus, we associate the robustness of a method with the quality of its solutions, i.e., the value of the objective function of the approximate solution it delivers. For a given problem, let $f_1,\dots,f_q$ be the functional value found by the methods $M_1,\dots,M_q$ being compared. Given a tolerance $\ftol>0$, we say that $f_i$ is equivalent to the best value found if
$f_i \leq f_{\min} + \ftol \max\{1,|f_{\min}|\}$
where $f_{\min} = \min_{s=1,\dots,q}\{f_s\}$, or if $f_i \leq -10^{-12}$. If $f_i$ is equivalent to the best value found, we say that the method $M_i$ was successful. Otherwise, we say that it failed. The greater the number of successes of a method, the greater its robustness. We only analyze the efficiency of methods in problems where the solutions computed by both methods are considered equivalent. We use CPU time as a measure of efficiency and present the comparison of efficiency in the form of performance profiles~\cite{pp}.

Given $\ftol>0$, let $p$ be the number of problems in which methods $M_1, \dots, M_q$ being compared found equivalent solutions and let $t_{ij}$ be the CPU time of method $M_i$ when applied to problem $j$. In a performance profile, the curve $\Gamma_i(\tau)$ associated with method $M_i$ is given by
$\Gamma_i(\tau) = \# \left\{ j \in \{1,\dots,p\} \;|\; t_{ij} \leq \tau \min_{s=1,\dots,q} \{t_{sj} \}\} \right\} / p$
for $\tau \geq 1$. The value of $\Gamma_i(1)$ corresponds to the proportion of problems in which method $M_i$ was the fastest (including ties). Since only problems in which the methods find equivalent solutions are considered, for all $i$ there exists a finite value of $\tau$ such that $\Gamma_i(\tau)=1$. Another option would be to include problems in which the methods fail, considering $t_{ij}=+\infty$ if the method $M_i$ failed on problem $j$. In that case, for each $i$, there exists a finite $\bar \tau$ such that $\Gamma_i(\tau)$ is constant for all $\tau \geq \bar \tau$ and the value of $\Gamma_i(\bar \tau)$ can be understood as a measure of robustness of the method $M_i$. In this work, we evaluate the robustness first and restrict the performance profiles to evaluate the efficiency only.

\subsubsection{Evaluation of alternatives in Algorithm~P}

In Algorithm~P, we considered the standard values $\theta = 0.1$, $\rho = 10^{-4}$, $a_1 = 10^8$, $a_2 = 10^{-16}$, $\lambda_{\min}^{\mathrm{spg}} = 10^{-16}$, and $\lambda_{\max}^{\mathrm{spg}} = 10^{16}$ from the literature. See, for example, \cite{bmgencan}. For the parameter $m$ that limits the effort of the optional extrapolations, we considered $m \in \{0, 5, 10, 15, 20\}$. The case $m=0$ corresponds to no optional extrapolations at all. For the parameter $\epsilon^{\mathrm{\textsc{mr}}}_{\mathrm{ini}}$ that determines the tolerance for the solution of Newtonian linear systems, we considered $\epsilon^{\mathrm{\textsc{mr}}}_{\mathrm{ini}} \in \{ 10^{-1}, 10^{-2}, 10^{-3}, 10^{-4}, \epsilon \}$. Note that when $\epsilon^{\mathrm{\textsc{mr}}}_{\mathrm{ini}} = \epsilon^{\mathrm{\textsc{mr}}}_{\mathrm{end}}=\epsilon$ all Newtonian linear systems are solved to full precision. 

The five options for values of $m$ plus the five options for $\epsilon^{\mathrm{\textsc{mr}}}_{\mathrm{ini}}$ and the two options for the choice of direction $d_1^k$ when MINRES detects a non-positive curvature direction (see Section~\ref{sec311}) leave a total of fifty parameter combinations for Algorithm~P. The best combination of parameters was found by employing irace~\cite{irace}. The irace package implements the Iterated Race method for the automatic tuning of optimization algorithms, given a set of instances of an optimization problem. We used 20\% of the 475 problems as a training set. The problems were selected by ordering the unconstrained and box-constrained problems from smallest to largest by the number of variables, separately, and selecting one in each five in the two ordered sets. The irace package requires a scalar merit function to evaluate the performance of the method whose parameters are being calibrated. We ignored the final value of the objective function and considered CPU time as a performance metric, considering a time of 10 minutes if the stopping criterion of a small projected continuous gradient was not reached. The combination that was identified as the best by irace was $m=20$, $\epsilon^{\mathrm{\textsc{mr}}}_{\mathrm{ini}}=0.1$ and taking $d_1^k$ as the approximate solution $s^k$ (instead of the residue $r^k$) when MINRES detectes a non-positive curvature direction. The values of $m$ and $\epsilon^{\mathrm{\textsc{mr}}}_{\mathrm{ini}}$ coincide with values reported in the literature for similar situations~\cite{bmgencan}. The choice of $d_1^k$ coincides with the results of the preliminary experiments carried out to define Algorithm~P.

It is important to discuss the sensitivity of the method in relation to its parameter and algorithmic choice options. When comparing the $m=0$ and $m=20$ options, we see that the latter finds values of $f$ smaller than $-10^{12}$ (suggesting that the objective function may be unbounded from below) in three more problems (eight versus eleven) and finds better function values in thirty-one problems when considering tolerance $\ftol=0.1$. The variant with $m = 0$ finds values of $f$ equivalent to the best one in 439 problems, while the variant with $m = 20$ does the same in 470 problems. Of the problems in which the two variants found equivalent values of $f$ with tolerance $\ftol = 0.1$, the variant with $m = 0$ is faster in 49\% of the problems, while the variant with $m = 20$ is faster in 54\%. In conclusion, extrapolations increase the effectiveness of the method but have little impact on its average efficiency. Now, examine the options for choosing a search direction when MINRES identifies a non-positive curvature direction. The options are to use the approximate solution found by MINRES as the search direction or to use the residue as the search direction. Both options identify 11 values of $f$ that are smaller than $-10^{12}$. However, the first option identifies 469 values of $f$ that are considered equivalent to the best, while the second option identifies 446. Considering problems in which both options identify equivalent function values, the first option is faster in 57\%, while the second option is faster 45\%. In summary, choosing the approximate solution of the linear system as search direction is a more robust and efficient option. This practical observation contrasts with the fact that the second option guarantees a functional decrease $\mathcal{O}(\epsilon^{3/2})$. Regarding the tolerance required to solve linear Newtonian systems, we highlight the difference between the more relaxed option $\epsilon^{\mathrm{\textsc{mr}}}_{\mathrm{ini}}=0.1$, and the more stringent option $\epsilon^{\mathrm{\textsc{mr}}}_{\mathrm{ini}}=\epsilon^{\mathrm{\textsc{mr}}}_{\mathrm{end}}=\epsilon=10^{-8}$. Surprisingly, the two options produced very similar results. The two variants identified 11 cases in which $f$ appears to be unbounded below and found 470 and 461 better function values, respectively. Taking into account the cases in which they identified equivalent function values, the former variant was faster in 59\% of the cases, while the latter variant was faster in 45\% of the cases. In short, the two variants were very similar, with a slight advantage in robustness and efficiency for the variant in which Newtonian systems are solved with increasing accuracy.

\subsubsection{Evaluation of alternatives in Algorithm~T}

In Algorithm~T we considered $\theta = 0.1$, $\eta = 10^{-8}$, $\rho = 10^{-4}$, $\alpha = 10^{-8}$, and $\gamma = 1$. For the parameter~$m$ that limits the effort of the optional extrapolations, we considered $m \in \{0, 5, 10, 15, 20\}$. For the parameters $\eta_0$ and $\tau$ that determine the tolerance for the solution of Newtonian linear systems, we considered $\eta_0 \in \{ 10^{-1}, 10^{-2}, 10^{-3}, 10^{-4}, \eta \}$ and $\tau \in \{ 0.1, 0.3, 0.5, 0.7, 0.9 \}$. We found the best combination of parameters for Algorithm~T by using irace the same way we used to calibrate the parameters of Algorithm~P. The best configuration returned by irace was $m=20$, $\eta_0 = \eta$, and $\tau = 0.9$. Note that since $\eta_0=\eta$ means that the Newtonian linear systems are solved to full precision by MINRES, the value of $\tau$ has no effect in Algorithm~T.

The influence of the extrapolations in Algorithm~T was very similar to that already reported for Algorithm~P. On the other hand, the tolerance required in the solution of linear Newtonian systems deserves special mention. Here, we consider the variants with $\eta_0=0.1$ and $\eta_0=\eta=\epsilon=10^{-8}$. It should be noted that when the Newton direction is computed with higher precision, there is a higher chance that the search direction produces a functional descent $\mathcal{O}(\epsilon^{3/2})$. Conversely, when the Newton direction is computed with lower precision, there is a greater chance that the search direction will guarantee only a functional descent of order $O(\epsilon^2)$. Consequently, Algorithm~T is forced to make more iterations of both Algorithm~5 and Algorithm~6, resulting in a less efficient method. On the one hand, the variant with $\eta_0=0.1$ proved to be more robust, finding 469 better function values against 441 of the other variant. On the other hand, in the cases where the two variants found function values considered equivalent, the first was faster in 63\% of the cases while the second was faster in 42\% of the cases. In other words, the variant that solves linear Newtonian systems in a relaxed way makes more use of Newton iterations with regularization. This method is considered more robust, though less efficient. Because efficiency was prioritized when choosing parameters with irace, the variant with $\eta_0=\eta$ was selected.

\subsubsection{Algorithm~P versus Algorithm~T}

We end this section by comparing Algorithms~P and~T. Considering the 475 problems, Algorithms~P and~T stopped at the CPU time limit in 25 and 42 problems, found a function value less than or equal to $-10^{12}$ in 11 and 13 problems, and found a point with a gradient sup-norm less than or equal to $\epsilon$ in 401 and 371 problems, respectively. Regardless of this, considering all 475 problems, Table~\ref{tab:spg-cr} shows the comparison of the function values found, and Figure~\ref{fig:spg-cr} compares the efficiency of Algorithms~P and~T in those problems where both found equivalent function values with $\ftol=0.1$. The table shows that Algorithm~P is substantially more robust than Algorithm~T, since it finds a significantly larger number of better solutions, regardless of the tolerance considered to determine that functional values are equivalent. The figure shows that when both methods find equivalent functional values, Algorithm~P is slightly more efficient than Algorithm~T. A comparison with the version of Algorithm~T that uses $\eta_0=0.1$ would show that the algorithms are similar in robustness, but Algorithm~P is much more efficient.

\begin{table}[ht!]
\begin{center}
{\small
\begin{tabular}{ccccccccc}
\cline{2-9}
& \multicolumn{8}{c}{$\ftol$} \\ 
\hline
& $0.1$ & $10^{-2}$ & $10^{-3}$ & $10^{-4}$ & $10^{-5}$ & $10^{-6}$ & $10^{-7}$ & $10^{-8}$ \\
\hline
Algorithm~P & 465 & 461 & 457 & 454 & 451 & 450 & 447 & 441 \\
Algorithm~T & 412 & 399 & 388 & 380 & 377 & 369 & 369 & 369 \\
\hline
\end{tabular}}
\end{center}
\caption{Number of solutions equivalent to the best solution found by Algorithms~P and~T, as a function of the tolerance $\ftol \in \{10^{-1}, 10^{-2} ,\dots,10^{-8}\}$, considering all the 475 unconstrained problems and bound-constrained problems from the CUTEst collection.}
\label{tab:spg-cr}
\end{table}
\begin{figure}[ht!]
\begin{center}
\scalebox{1.0}{\input{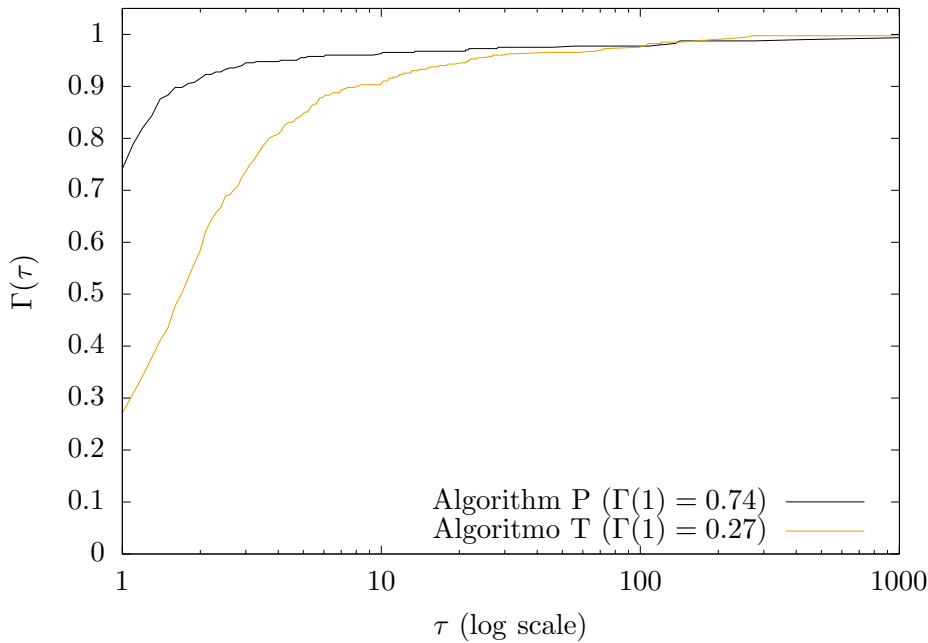}}
\end{center}
\caption{Performance profiles comparing the efficiency of Algorithms~P and~T on the 402 problems where the two methods found equivalent objective function values with tolerance $\ftol = 0.1$.} 
\label{fig:spg-cr}
\end{figure}

\subsection{Comparision of Algorithm~P and Gencan} \label{sec33}

In this section we compare Algorithm~P with Gencan (included in Algencan 3.1.1 and freely available at \url{http://www.ime.usp.br/~tango/}). Gencan is an active set method for bound-constrained minimization, introduced in~\cite{bmgencan}. Algencan~\cite{abmstango,bmbook,bmcomper}, an augmented Lagrangian method for nonlinear programming, uses Gencan to solve its subproblems. Gencan is an active set method whose general framework is exactly the same as that described by Algorithm~1. That is, it uses exactly the same criteria as Algorithm~P to decide whether the next iteration should be within the current face or whether the current face should be abandoned. If the current face should be abandoned, Gencan also uses an SPG iteration as described in Algorithm~3. For iterations within a face, when matrix-factorizations are not allowed, Gencan, like Algorithm~P, uses truncated Newton with line search. The difference is that Newtonian linear systems are solved with conjugate gradients. In conjugate gradients, if Hessians are not available, the Hessian vector products are approximated by differences of gradients. In the present work, we are assuming that Hessians are available and, therefore, Gencan as well as Algorithms~P and~T use true Hessian vector products in conjugate gradients and MINRES. Gencan's truncated Newton inspired and shares with Algorithm~2 the way to calculate the tolerance with which linear systems should be solved, what to do when a non-positive curvature direction is detected, and how to decide whether to attempt extrapolations or not. That is, the only relevant difference between Gencan and Algorithm~P is that the former uses conjugate gradients and the latter uses MINRES to solve Newtonian linear systems. It is important to mention that the comparison presented in~\cite{bg} ranked Gencan among the most efficient and robust methods for bound-constrained minimization, in a comparison that included ASA-CG~\cite{asacg}, Ipopt~\cite{ipopt}, Lancelot~B~\cite{Conn1992}, L-BFGS-B~\cite{lbfgsb3}, SPG~\cite{Birgin2000} and fmincon~\cite{Coleman1994,Coleman1996}.

We run Gencan with all its default parameters and the same stopping criterion already mentioned for Algorithms~P and~T, i.e., $\| \nabla_{\Omega} f(x) \|_{\infty} \leq \epsilon$ with $\epsilon = 10^{-8}$. Considering the 475 problems, Algorithm~P and Gencan stopped at the CPU time limit in 25 and 33 problems, found a function value less than or equal to $-10^{12}$ in 11 and 12 problems, and found a point with a gradient sup-norm less than or equal to $\epsilon$ in 401 and 364 problems, respectively. Regardless of this, considering all the 475 problems, Table~\ref{tab:spg-cg} shows the comparison of the function values found, and Figure~\ref{fig:spg-cg} compares the efficiency of the two variants in those problems where both found equivalent function values with $\ftol=0.1$. The table and the figure show that Algorithm~P is slightly more robust and significantly more efficient than Gencan. When we consider only the 313 unconstrained problems in the CUTEst collection, the results are qualitatively equivalent to those shown in \cite{Liu2023}. In that study, the Newton-MR method, from which Algorithms P and T originated, was found to be more robust and efficient than several variations of Newton's method that use conjugate gradients to solve Newtonian linear systems.

\begin{table}[ht!]
\begin{center}
{\small
\begin{tabular}{ccccccccc}
\cline{2-9}
& \multicolumn{8}{c}{$\ftol$} \\ 
\hline
& $0.1$ & $10^{-2}$ & $10^{-3}$ & $10^{-4}$ & $10^{-5}$ & $10^{-6}$ & $10^{-7}$ & $10^{-8}$ \\
\hline
Algorithm~P & 462 & 453 & 448 & 443 & 438 & 436 & 435 & 428 \\
Gencan      & 447 & 445 & 437 & 430 & 427 & 423 & 422 & 416 \\
\hline
\end{tabular}}
\end{center}
\caption{Number of solutions equivalent to the best solution found by Algorithm~P and Gencan, as a function of the tolerance $\ftol \in \{10^{-1}, 10^{-2} ,\dots,10^{-8}\}$, considering all the 475 unconstrained problems and bound-constrained problems from the CUTEst collection.}
\label{tab:spg-cg}
\end{table}

\begin{figure}[ht!]
\begin{center}
\scalebox{1.0}{\input{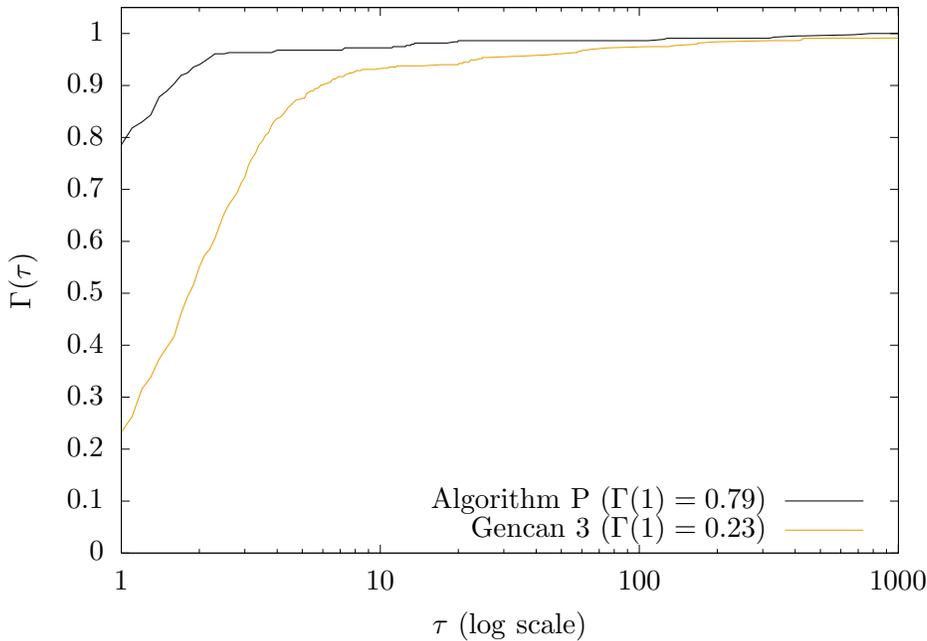}}
\end{center}
\caption{Performance profiles comparing the efficiency of Algorithm~P and Gencan on the 434 problems where the two methods found equivalent objective function values with tolerance $\ftol = 0.1$.} 
\label{fig:spg-cg}
\end{figure}

\section{Conclusion} \label{sec4}

Recent work has analyzed the practical and theoretical properties of the well-known MINRES method for solving linear systems, particularly in the context of a truncated Newton method (Newton-MR) for unconstrained minimization. In this paper, we extended the Newton-MR method in two distinct ways. In one approach, we preserved the worst-case complexity of $\mathcal{O}(\epsilon^{-3/2})$ exhibited by Newton-MR for unconstrained minimization. In the other approach, inspired by Gencan and guided by numerical evaluations of various alternatives, we developed an extension of Newton-MR for bound-constrained minimization with worst-case complexity of $\mathcal{O}(\epsilon^{-2})$. Numerical experiments demonstrated that the latter approach is more robust and efficient than the former, when considering both unconstrained and bound-constrained problems from the CUTEst collection. A similar conclusion is reached when only unconstrained problems are considered. On the one hand, it can be argued that worst-case complexity does not always accurately reflect a method's practical performance. On the other hand, it is important to note that the method with lower complexity requires stronger assumptions than the method with higher complexity. These stronger assumptions are difficult to verify in practice. The best of the two methods was also compared with Gencan, a method with similar characteristics but that solves linear systems using conjugate gradients. The new method proved to be more robust and efficient. In future work, it remains to be seen whether this advantage holds when the method is used to solve subproblems in an augmented Lagrangian method. 

\bibliographystyle{plain}

\bibliography{bgm2025}

\end{document}